\documentclass[10pt,reqno,final]{article}

\usepackage{amsmath,amsfonts,amssymb,amsthm,version}
\usepackage{mathrsfs,fancybox,pifont}
\usepackage{graphicx}
\usepackage{url,hyperref}
\usepackage[notcite,notref]{showkeys}
\usepackage{color}
\usepackage{subfigure,multirow}
\usepackage{amsmath,amssymb}
\usepackage{graphicx}
\usepackage{float}
\usepackage{booktabs}
\usepackage{multirow}
\usepackage{tabularx}
\usepackage{bm}
\usepackage{subfigure,multirow}
\usepackage{epstopdf}
\usepackage{cases}
\usepackage{mathtools}
\usepackage{algorithm,algorithmic}
\usepackage{authblk}
\usepackage{lipsum}

\allowdisplaybreaks

\theoremstyle{plain}
\newtheorem{theorem}{Theorem}[section]

\newtheorem{assumption}{Assumption}[section]

\theoremstyle{definition}

\theoremstyle{remark}
\newtheorem{remark}{Remark}[section]

\title{A Regularized Auxiliary Variable (RAV) Approach for Gradient Flows}
\author{Zhaoyang Wang\thanks{School of Mathematical Sciences, Laboratory of Mathematics and Complex Systems, MOE, Beijing Normal University, Beijing 100875, China; Research Center for Mathematics, Advanced Institute of Natural Sciences, Beijing Normal University, Zhuhai, Guangdong 519087, China. (zhaoyang584520@163.com)} \, and
Ping Lin\thanks{Corresponding author. Division of Mathematics, University of Dundee, Dundee DD1 4HN, United Kingdom (p.lin@dundee.ac.uk)}}

\affil{}

\date{}

\begin{document}
\maketitle

\begin{abstract}
In this paper, we propose a regularized auxiliary variable (RAV) approach and construct accurate and robust time-discrete schemes for a large class of gradient flows. By introducing an auxiliary variable $r=0$ and constructing an auxiliary equation that naturally fits into the energy relation, the numerical solution $r^{n+1}$ of the auxiliary variable is corrected at each time step to preserve consistency with the original system. The developed RAV scheme satisfies unconditional energy stability with respect to the original variables, and in certain cases the original energy law can be directly recovered. Furthermore, we obtain a uniform bound on the norm of the numerical solution, which allows us to establish the optimal error estimate in $L^\infty(0,T;H^2)$ for the second-order scheme without any restriction on the time step. We present ample numerical results, including comparisons with the scalar auxiliary variable (SAV) approach, to demonstrate the accuracy and effectiveness of the proposed RAV approach.

\medskip
\noindent{\bf Keywords}: Gradient flows, regularization method, auxiliary variables, stability, error analysis

\medskip
\end{abstract}

\section{Introduction}
\label{section1}
Gradient flows are driven by free energy and constitute a class of evolutionary models that are ubiquitous in the natural sciences and engineering. Many important partial differential equations, whether describing interfacial evolution \cite{lowengrub1998quasi, allen1979microscopic, wang2025stability, anderson1998diffuse}, thin film dynamics \cite{chen2008phase, shen2012second}, or the evolution of polymer and liquid crystal systems \cite{lin2006simulations, wang2025class}, can be interpreted as gradient flows endowed with an energy dissipation structure.

From the perspective of model construction, gradient flows are typically derived from the total free energy and its variational formulation, subject to the constraint imposed by the second law of thermodynamics. Consider the free energy functional $E[\phi(\bm{x})]=\frac{1}{2}(\phi, \mathcal{L}\phi)+(F(\phi),1)$, where $\mathcal{L}$ is a symmetric non-negative linear operator and $F$ is the energy density function. Based on this formulation, the general structure of the gradient flow can be expressed as
\begin{equation}\label{the1-1}
\begin{split}
&\frac{\partial \phi}{\partial t}=\mathcal{G}\mu, \\
&\mu=\frac{\delta E}{\delta \phi}=\mathcal{L}\phi+F'(\phi),
\end{split}
\end{equation}
with suitable boundary conditions. Here, $\mathcal{G}$ is a non-positive symmetric operator that characterizes the dissipative mechanism of the system. For instance, $\mathcal{G}=-I$ corresponds to the $L^2$ gradient flow, while $\mathcal{G}=\Delta$ corresponds to the $H^{-1}$ gradient flow. The gradient flow system admits the following energy dissipation law:
\begin{equation}\label{the1-2}
\begin{split}
\frac{d}{dt}E[\phi]=(\frac{\delta E}{\delta \phi}, \frac{\partial \phi}{\partial t})=(\mu, \mathcal{G}\mu)\leq 0.
\end{split}
\end{equation}
Therefore, the development of efficient and accurate numerical schemes that preserve this energy dissipation law is of critical importance. 

In recent years, the SAV approach \cite{shen2019new, shen2018scalar} has become increasingly popular due to its ability to construct efficient and unconditionally energy stable schemes for a large class of gradient flows. The key idea is to introduce a scalar variable $r(t)=\sqrt{\int_\Omega F(\phi) d\bm{x}+C_0}$ ($C_0$ is chosen to ensure that $\int_\Omega F(\phi) d\bm{x}+C_0>0$) to obtain an auxiliary ordinary differential equation (ODE)
\begin{equation}\label{the1-3}
\begin{split}
\frac{d}{dt}r(t)=\frac{1}{2\sqrt{\int_\Omega F(\phi) d\bm{x}+C_0}}\int_\Omega  F'(\phi)\frac{\partial \phi}{\partial t} \ d\bm{x}.
\end{split}
\end{equation}

Let $T$ be the final time of computation and $N$ a positive integer. By setting $\Delta t=T/N$, and letting $(\cdot)^{n}$ denote the numerical approximation of a specific variable at $t=n\Delta t$ for $n\leq N$, the first-order scheme is then constructed with the nonlinear term $F(\phi)$ treated explicitly, and is given by
\begin{equation}\label{the1-4}
\begin{split}
&\frac{\phi^{n+1}-\phi^n}{\Delta t}=\mathcal{G}\mu^{n+1}, \\
&\mu^{n+1}=\mathcal{L}{\phi^{n+1}}+\frac{r^{n+1}}{\sqrt{\int_\Omega F(\phi^n) \ d\bm{x}+C_0}}F'(\phi^n),\\
&\frac{r^{n+1}-r^n}{\Delta t}=\frac{1}{2\sqrt{\int_\Omega F(\phi^n) d\bm{x}+C_0}}\int_\Omega F'(\phi^n) \frac{\phi^{n+1}-\phi^{n}}{\Delta t}\ d\bm{x}.
\end{split}
\end{equation}
One obtains the modified discrete energy dissipation law:
\begin{equation}\label{the1-5}
\begin{split}
\frac{1}{2}(\phi^{n+1},\mathcal{L}\phi^{n+1})+|r^{n+1}|^2-\left(\frac{1}{2}(\phi^{n},\mathcal{L}\phi^{n})+|r^{n}|^2\right)\leq 0.
\end{split}
\end{equation}
Furthermore, it requires solving only linear systems with constant coefficients at each time step. Thanks to these advantages, the SAV method has been successfully applied to a wide range of gradient flow problems \cite{wang2021second, yang2021novel, yang2022linear, feng2021high, crouseilles2025semi}.

As a variant of the SAV approach, a generalized SAV (GSAV) method \cite{huang2022new} capable of achieving high-order accuracy has been proposed, which employs the dynamical equation (\ref{the1-2}) to ensure energy stability. By setting $r(t)=E[\phi]+C_0>0$, the BDF2 GSAV scheme is as follows:
\begin{equation}\label{the1-6}
\begin{split}
&\frac{3\overline{\phi}^{n+1}-4\phi^n+2\phi^{n-1}}{2\Delta t}=\mathcal{G}\mu^{n+1}, \ \mu^{n+1}=\mathcal{L}{\phi^{n+1}}+F'(2\overline{\phi}^{n}-\overline{\phi}^{n-1}),\\
&\frac{r^{n+1}-r^n}{\Delta t}=\frac{r^{n+1}}{E[\overline{\phi}^{n+1}]+C_0}(\mu^{n+1}, \mathcal{G}\mu^{n+1}),\\
&\xi^{n+1}=\frac{r^{n+1}}{E[\overline{\phi}^{n+1}]+C_0}, \\
&\phi^{n+1}=\eta^{n+1}\overline{\phi}^{n+1} \ \text{with} \ \eta^{n+1}=\xi^{n+1}(2-\xi^{n+1}).
\end{split}
\end{equation}
A more general modified energy law can be readily derived: $r^{n+1}-r^{n}\leq 0$.

It can be seen from (\ref{the1-3}) that the SAV method directly replaces the algebraic expression of the auxiliary variable with its time derivative to formulate an ODE. This direct reduction of the index via differentiation may lead to the drift problem (see e.g. \cite{lin1997sequential}), in which the original algebraic form of the auxiliary variable cannot be accurately maintained as time goes on. This inconsistency introduces errors during the computation, and although the schemes preserve the modified energy law, they do not necessarily preserve the original energy law. The relaxation strategy proposed in \cite{jiang2022improving} and \cite{zhang2022generalized} can effectively remedy such inconsistency after each computational step. However, the schemes (\ref{the1-4}) and (\ref{the1-6}) are not fully consistent with the original equation, since the term $\frac{r^{n+1}}{\sqrt{\int_\Omega F(\phi^n) \ d\bm{x}+C_0}}$ in (\ref{the1-4}) or $\frac{r^{n+1}}{E[\overline{\phi}^{n+1}]+C_0}$ in (\ref{the1-6}) does not approximate $1$ accurately enough (especially when $\Delta t$ is large). This leads to the computed original variable $\phi^{n+1}$ and the original energy $E(\phi^{n+1})$ being numerically inaccurate. 

Inspired by the regularization method for the Navier-Stokes equations \cite{lin1997sequential, wanglin2026}, we propose a regularized auxiliary variable (RAV) approach to construct efficient and accurate energy stable schemes for general gradient flow problems. Specifically, we set the auxiliary variable $r=0$ and use the energy relation to perform verification and correction at each time step in the numerical computation, so that the analytic relation can be enforced in a stable manner under the premise of ensuring the energy stability of the scheme. Since the constructed RAV scheme preserves, as much as possible, the consistency between the auxiliary variable $r^{n+1}$ and its original value $0$, it is more temporally stable in the differential–algebraic context than direct index reduction by differentiation. Furthermore, the scheme only requires solving one linear system with constant coefficients at each time step, and its computational cost is the same as that of the implicit–explicit (IMEX) scheme (\ref{the1-6}). More importantly, we are able to establish optimal error estimates in $L^\infty(0,T;H^2)$ without any restriction on the time step. To the best of our knowledge, this is the first result of such an IMEX scheme for gradient flows.

The rest of paper is organized as follows. In Section \ref{section2}, we construct second-order and higher-order RAV schemes and rigorously prove their unconditional stability. In Section \ref{section3}, we carry out a rigorous error analysis for the second‑order RAV scheme without any restriction on the time step. In Section \ref{section4}, various numerical examples are presented to demonstrate the performance of the proposed method and to provide comparisons with the SAV method. Some conclusions and remarks are given in Section \ref{section5}.

\subsection*{Notation}
For domain $\Omega$ in $\mathbb{R}^{d} \ (d=2,3)$ and $1\leq p\leq\infty$, we use the standard notation for the Banach space $L^{p}(\Omega)$ and the Sobolev space $W^{k,p}(\Omega)$ or $H^p(\Omega)$ and $H^p_0(\Omega)$. The symbol $(\cdot,\cdot)$ indicates the standard scalar product in $L^2$. Throughout this paper, the letter $C$ denotes a generic positive constant, with or without subscript, its value may change from one line of an estimate to the next. We will write the dependence of the constant on parameters explicitly if it is essential.

\section{The regularized auxiliary variable (RAV) approach and its schemes}
\label{section2}
In this section, we first propose the RAV approach for the gradient flow of a single function and construct the corresponding second-order and higher-order schemes, and then extend it to gradient flows involving several functions. 

To clearly illustrate the RAV method, we consider here a typical free energy functional 
\begin{equation}\nonumber
\begin{split}
E[\phi]=\int_\Omega \left(\frac{\epsilon^2}{2}\left|\nabla \phi\right|^2+F(\phi)\right)\ d\bm{x},
\end{split}
\end{equation}
and the corresponding gradient flow in $H^{-1}$:
\begin{equation}\label{the2-1}
\begin{split}
&\frac{\partial \phi}{\partial t}=\Delta \mu,\\
&\mu=\frac{\delta E}{\delta \phi}=-\epsilon^2\Delta \phi+F'(\phi),
\end{split}
\end{equation}
subject to either periodic boundary conditions or $\frac{\partial \phi}{\partial \bm{n}}|_{\partial\Omega}=\frac{\partial \mu}{\partial \bm{n}}|_{\partial\Omega}=0$.

\subsection{Gradient flows of a single function}
Based on the ideas of auxiliary variables \cite{shen2018scalar, shen2019new} and regularization method \cite{lin1997sequential, wanglin2026}, let $r(t)$ be a time-dependent auxiliary variable with $r|_{t=0}=0$. We then consider the following equation:
\begin{equation}\label{the2-2}
\begin{split}
 \frac{d}{dt}r(t)=\int_\Omega F'(\phi)\phi_t \ d\bm{x}-\frac{d}{dt}\left(\int_\Omega  F(\phi) \ d\bm{x}\right)\equiv 0.
\end{split}
\end{equation}
It is easy to see that $r\equiv 0$ for $t>0$. We assume that the energy $E[\phi]$ is bounded from below, i.e., $\tilde{E}[\phi]=E[\phi]+C_0>0$. The system (\ref{the2-1}) can be equivalently reformulated as:
\begin{subequations}\label{the2-3}
	\begin{align}
&\frac{\partial \phi}{\partial t}=\Delta \mu,\\
&\mu=-\epsilon^2\Delta \phi+F'(\overline{\phi}), \\
&\frac{d}{dt}r(t)=\int_\Omega F'(\overline{\phi})\phi_t \ d\bm{x}-\frac{d}{dt}\left(\int_\Omega  F(\phi) \ d\bm{x}\right), \label{the2-3c}\\
&\overline{\phi}=\xi\phi=\frac{\tilde{E}[\phi]+r}{\tilde{E}[\phi]}\phi.
\end{align}
\end{subequations}
The second-order midpoint scheme is given by:
\begin{subequations}\label{the2-4}
	\begin{align}
&\frac{\phi^{n+1}-\phi^{n}}{\Delta t}=\Delta \mu^{n+\frac{1}{2}}, \label{the2-4a}\\
&\mu^{n+\frac{1}{2}}=-\epsilon^2\Delta \phi^{n+\frac{1}{2}}+\lambda \phi^{n+\frac{1}{2}}+F'((\overline{\phi}^*)^{n+\frac{1}{2}})-\lambda(\overline{\phi}^*)^{n+\frac{1}{2}}, \label{the2-4b}\\
&r^{n+1} = 
\begin{cases}  
0, & Q^{n+1} \geq 0, \\ 
Q^{n+1}, & Q^{n+1} < 0, 
\end{cases} \label{the2-4c}\\
&\overline{\phi}^{n+1}=\xi^{n+1}\phi^{n+1}:=\left(\frac{\tilde{E}[\phi^{n+1}]+r^{n+1}}{\tilde{E}[\phi^{n+1}]}\right)\phi^{n+1}, \label{the2-4d}
\end{align}
\end{subequations}
where $\lambda>0 $ is a stabilization parameter \cite{chen2018regularized}, and $(\overline{\phi}^*)^{n+\frac{1}{2}}=\frac{3}{2}\overline{\phi}^n-\frac{1}{2}\overline{\phi}^{n-1}$. The numerical variable
\begin{equation}\nonumber
\begin{split}
&Q^{n+1}=r^n+\int_\Omega \left(\left(F'((\overline{\phi}^*)^{n+\frac{1}{2}})-\lambda(\overline{\phi}^*)^{n+\frac{1}{2}}\right)(\phi^{n+1}-\phi^n)\right) \ d\bm{x}\\
&-\left(\int_\Omega (F(\phi^{n+1})-\frac{\lambda}{2}(\phi^{n+1})^2) \ d\bm{x}-\int_\Omega (F(\phi^{n})-\frac{\lambda}{2}(\phi^{n})^2) \ d\bm{x}\right)+\Delta t\Vert \nabla\mu^{n+\frac{1}{2}}\Vert_{L^2}^2,
\end{split}
\end{equation}
with $r^{0}=0$. This second-order scheme requires initialization by a first-order scheme for the first step, and efficiently implemented according to the following steps. Given $\phi^0, \phi^1, ...,\phi^n$, $\overline{\phi}^0, \overline{\phi}^1, ..., \overline{\phi}^n$, and $r^0, r^1, ..., r^n$, we compute $\phi^{n+1}$, $\overline{\phi}^{n+1}$ and $r^{n+1}$:

\textbullet \ \textbf{Step 1}: Compute the solution $(\phi^{n+1}, \mu^{n+\frac{1}{2}})$ by using (\ref{the2-4a}) and (\ref{the2-4b}):
\[ 
\begin{pmatrix}
\frac{1}{\Delta t}I & -\Delta \\ 
-\frac{\lambda}{2}I+\frac{\epsilon^2}{2}\Delta & I
\end{pmatrix} 
\begin{pmatrix} 
\phi^{n+1} \\ 
\mu^{n+\frac{1}{2}} 
\end{pmatrix}
=
\begin{pmatrix} 
\frac{1}{\Delta t}\phi^{n} \\ 
\frac{\lambda}{2}\phi^n-\frac{\epsilon^2}{2}\phi^n+F'((\overline{\phi}^*)^{n+\frac{1}{2}})-\lambda(\overline{\phi}^*)^{n+\frac{1}{2}}
\end{pmatrix}.
\]

\textbullet \ \textbf{Step 2}: Compute $Q^{n+1}$ and subsequently obtain $r^{n+1}$ and $\overline{\phi}^{n+1}=\xi^{n+1}\phi^{n+1}:=\left(\frac{\tilde{E}[\phi^{n+1}]+r^{n+1}}{\tilde{E}[\phi^{n+1}]}\right)\phi^{n+1}$.

\begin{remark}
Following the Baumgarte stabilization formulation for differential-algebraic equations (DAEs) \cite{baumgarte1972stabilization, lin1997sequential} and the relaxation strategy for gradient flows \cite{zhang2022generalized, jiang2022improving}, we introduce a dissipation term $\Vert \nabla\mu^{n+\frac{1}{2}}\Vert_{L^2}^2$ into $Q^{n+1}$, which ensures numerical stability (see Theorems \ref{theorem2-2} and \ref{theorem2-3}) while keeping $r^{n+1}$ as close as possible to its original value $0$ when $Q^{n+1}<0$. It can be observed that the non‑positive auxiliary variable $r^{n+1}$ depends on the sign of $Q^{n+1}$, and its deviation from the original value $0$ does not accumulate significantly over the entire computational time interval as in the SAV method. Therefore, the RAV scheme is able to better preserve the structure of the original system.
\end{remark}

\begin{theorem}
\label{theorem2-2}
(Energy stability)
For $Q^{n+1}\geq 0$, the second-order RAV scheme (\ref{the2-4}) is unconditionally energy stable in the sense that the following discrete energy law holds:
\begin{equation}\label{the2-5}
\begin{split}
E[\phi^{n+1}]-(E[\phi^{n}]+r^{n})=E[\phi^{n+1}]+r^{n+1}-(E[\phi^{n}]+r^{n})\leq 0.
\end{split}
\end{equation}
When $Q^{n+1}<0$, we have
\begin{equation}\label{the2-6}
\begin{split}
E[\phi^{n+1}]+r^{n+1}-(E[\phi^{n}]+r^{n})= 0.
\end{split}
\end{equation}
Moreover, if $Q^{n+1}-r^{n}\geq 0$, the scheme satisfies the original energy dissipation law $E[\phi^{n+1}]-E[\phi^{n}]\leq 0$.
\end{theorem}

\begin{proof}
Taking the inner product of (\ref{the2-4a}) and (\ref{the2-4b}) with $\Delta t\mu^{n+\frac{1}{2}}$ and $\phi^{n+1}-\phi^n$, respectively, we obtain
\begin{equation}\label{the2-7}
\begin{split}
&\frac{\lambda}{2}\left(\Vert\phi^{n+1}\Vert_{L^2}^2-\Vert\phi^n\Vert_{L^2}^2\right)+\frac{\epsilon^2}{2}\left(\Vert \nabla \phi^{n+1}\Vert_{L^2}^2-\Vert \nabla \phi^{n}\Vert_{L^2}^2\right)\\
&+\int_\Omega \left(\left(F'((\overline{\phi}^*)^{n+\frac{1}{2}})-\lambda(\overline{\phi}^*)^{n+\frac{1}{2}}\right)(\phi^{n+1}-\phi^n)\right) \ d\bm{x}=-\Delta t\Vert \nabla \mu^{n+\frac{1}{2}}\Vert_{L^2}^2.
\end{split}
\end{equation}
It follows that
\begin{equation}\label{the2-8}
\begin{split}
E[\phi^{n+1}]-(E[\phi^{n}]+r^{n})+Q^{n+1}=0.
\end{split}
\end{equation}
We can directly derive (\ref{the2-5}) and (\ref{the2-6}) with the help of (\ref{the2-4c}). Furthermore, the original energy dissipation law is obtained in the case $Q^{n+1}-r^{n}\geq 0$.
\end{proof}

\begin{theorem}
\label{theorem2-3}
For the modified energy $\tilde{E}[\phi^{n}]+r^{n}=E[\phi^{n}]+r^{n}+C_0$, we have
\begin{equation}\label{the2-9}
\begin{split}
0<\tilde{E}[\phi^{n}]+r^{n}\leq \tilde{E}[\phi^{0}],
\end{split}
\end{equation}
and $0\leq \xi^{n}\leq 1$. Furthermore, if $\int_\Omega F(\phi) \ d\bm{x}$ is bounded from below, there exists constants $M>0$ and $M_T>0$ such that
\begin{equation}\label{the2-10}
\begin{split}
\Vert \overline{\phi}^{n}\Vert_{H^1}\leq M, \quad \Vert \phi^{n}\Vert_{H^2}+\Vert \overline{\phi}^{n}\Vert_{H^2}\leq M_T,
\end{split}
\end{equation}
where $M$ is independent of $T$, and $M_T$ may depend on $T$.
\end{theorem}

\begin{proof}
By Theorem \ref{theorem2-2}, we directly obtain $\tilde{E}[\phi^{n}]+r^{n}\leq \tilde{E}[\phi^{0}]$. Given $\tilde{E}[\phi^{n}]+r^{n}>0$, we have $\tilde{E}[\phi^{n+1}]+r^{n+1}=\tilde{E}[\phi^{n+1}]>0$ for $Q^{n+1}\geq 0$, and $\tilde{E}[\phi^{n+1}]+r^{n+1}=\tilde{E}[\phi^{n}]+r^{n}>0$ for $Q^{n+1}<0$.

Note that $r^{n}\leq 0$, so we have
\begin{equation}\label{the2-11}
\begin{split}
0\leq \xi^{n}= \frac{\tilde{E}[\phi^{n}]+r^{n}}{\tilde{E}[\phi^{n}]}\leq 1.
\end{split}
\end{equation}

Here, we assume $\int_\Omega F(\phi) \ d\bm{x}+C_0>1$ without loss of generality. It follows from (\ref{the2-4d}) that
\begin{equation}\label{the2-12}
\begin{split}
&\Vert \nabla \overline{\phi}^n\Vert_{L^2}^2=\left(\frac{\tilde{E}[\phi^{n}]+r^{n}}{\frac{\epsilon^2}{2}\Vert \nabla {\phi}^n\Vert_{L^2}^2+\int_\Omega F({\phi}^n) \ d\bm{x}+C_0}\right)^2\Vert \nabla \phi^n\Vert_{L^2}^2\\
&\leq \left(\frac{\tilde{E}[\phi^{0}]}{\frac{\epsilon^2}{2}\Vert\nabla \phi^n\Vert_{L^2}^2+1}\right)^2\Vert\nabla \phi^n\Vert_{L^2}^2\leq \frac{2}{\epsilon^2}(\tilde{E}[\phi^{0}])^2.
\end{split}
\end{equation}
Note that $|\int_\Omega \overline{\phi}^n \ d\bm{x}|=\xi^n|\int_\Omega \phi^n \ d\bm{x}|\leq |\int_\Omega \phi^0 \ d\bm{x}|$. By the Poincar\'{e} inequality,
\begin{equation}\label{the2-13}
\begin{split}
\Vert \overline{\phi}^n\Vert_{L^2}^2\leq C_\Omega\left( \Vert \nabla \overline{\phi}^n\Vert_{L^2}^2+\left|\int_\Omega \overline{\phi}^n \ d\bm{x}\right|\right)\leq C_\Omega\left( \Vert \nabla \overline{\phi}^n\Vert_{L^2}^2+\left|\int_\Omega \phi^0 \ d\bm{x}\right|\right).
\end{split}
\end{equation}
Therefore, we have $\Vert \overline{\phi}^{n}\Vert_{H^1}\leq M$.

Following arguments of Lemma 2.3 in \cite{shen2018convergence}, we can obtain 
\begin{equation}\label{the2-14}
\begin{split}
\Vert \nabla\left(F'((\overline{\phi}^*)^{n+\frac{1}{2}})-\lambda(\overline{\phi}^*)^{n+\frac{1}{2}}\right)\Vert_{L^2}^2\leq C+C_\varepsilon\left(\Vert \nabla \Delta\overline{\phi}^n\Vert_{L^2}^2+\Vert \nabla \Delta\overline{\phi}^{n-1}\Vert_{L^2}^2\right),
\end{split}
\end{equation}
and 
\begin{equation}\label{the2-15}
\begin{split}
\Vert \Delta \left(F'((\overline{\phi}^*)^{n+\frac{1}{2}})-\lambda(\overline{\phi}^*)^{n+\frac{1}{2}}\right)\Vert_{L^2}^2\leq C+C_\varepsilon\left(\Vert  \Delta^2\overline{\phi}^n\Vert_{L^2}^2+\Vert \Delta^2\overline{\phi}^{n-1}\Vert_{L^2}^2\right).
\end{split}
\end{equation}
Here, $C_\varepsilon$ is a positive constant sufficiently small.

Combining (\ref{the2-4a}) with (\ref{the2-4b}), and taking the inner product with $\Delta t\Delta^2(\phi^{n+1}+\phi^{n})$ leads to
\begin{equation}\label{the2-16}
\begin{split}
&\Vert \Delta \phi^{n+1}\Vert_{L^2}^2-\Vert \Delta \phi^{n}\Vert_{L^2}^2+\frac{\lambda\Delta t}{2}\left(\Vert \nabla\Delta \phi^{n+1}\Vert_{L^2}^2+\Vert \nabla\Delta \phi^{n}\Vert_{L^2}^2\right)+\frac{\epsilon^2\Delta t}{2}\Vert \Delta^2 (\phi^{n+1}+\phi^{n})\Vert_{L^2}^2\\
&\leq \Delta t\left|\left(\Delta(F'((\overline{\phi}^*)^{n+\frac{1}{2}})-\lambda(\overline{\phi}^*)^{n+\frac{1}{2}}), \Delta^2(\phi^{n+1}+\phi^{n})\right)\right|\\
&\leq C\Delta t+C_\varepsilon\Delta t(\Vert  \Delta^2\overline{\phi}^n\Vert_{L^2}^2+\Vert \Delta^2\overline{\phi}^{n-1}\Vert_{L^2}^2)+\frac{\epsilon^2\Delta t}{4}\Vert \Delta^2 (\phi^{n+1}+\phi^{n})\Vert_{L^2}^2.
\end{split}
\end{equation}
Summing over $n$, it follows that
\begin{equation}\label{the2-17}
\begin{split}
\Vert \Delta \phi^{n+1}\Vert_{L^2}^2+\Delta t\sum\limits_{k=0}^{n}\Vert \nabla\Delta \phi^{k+1}\Vert_{L^2}^2\leq M_T.
\end{split}
\end{equation}

Note that
\begin{equation}\label{the2-18}
\begin{split}
\int_\Omega &\left(\left(F'((\overline{\phi}^*)^{n+\frac{1}{2}})-\lambda(\overline{\phi}^*)^{n+\frac{1}{2}}\right)(\phi^{n+1}-\phi^n)\right) \ d\bm{x}\\
&\leq \frac{\Delta t}{2}\Vert \nabla \mu^{n+\frac{1}{2}}\Vert_{L^2}^2+C\Delta t\left(1+\Vert \nabla \Delta\overline{\phi}^n\Vert_{L^2}^2+\Vert \nabla \Delta\overline{\phi}^{n-1}\Vert_{L^2}^2\right).
\end{split}
\end{equation}
By (\ref{the2-7}), we estimate
\begin{equation}\label{the2-19}
\begin{split}
&\frac{\lambda}{2}\Vert\phi^{n+1}\Vert_{L^2}^2+\frac{\epsilon^2}{2}\Vert \nabla \phi^{n+1}\Vert_{L^2}^2+\frac{\Delta t}{2}\sum\limits_{k=0}^{n}\Vert \nabla \mu^{k+\frac{1}{2}}\Vert_{L^2}^2\\
&\leq C+C\Delta t\left(\sum\limits_{k=0}^{n}\Vert \nabla\Delta \phi^{k}\Vert_{L^2}^2+\sum\limits_{k=0}^{n-1}\Vert \nabla\Delta \phi^{k}\Vert_{L^2}^2\right)\\
&\leq M_T,
\end{split}
\end{equation}
which gives
\begin{equation}\label{the2-20}
\begin{split}
&\Vert \phi^n\Vert_{H^2}^2\leq M_T,\\
& \Vert \overline{\phi}^n\Vert_{H^2}^2=(\xi^n)^2\Vert \phi^n\Vert_{H^2}^2\leq \Vert \phi^n\Vert_{H^2}^2\leq M_T.
\end{split}
\end{equation}

This completes the proof.
\end{proof}

\begin{remark}
The modified energy $E[\phi^{n}]+r^{n}$ that we define clearly characterizes its relation to the original energy $E[\phi^{n}]$ and yields a uniform bound on $\Vert \overline{\phi}^n\Vert_{H^1}$. Furthermore, the numerical stability derived in Theorem \ref{theorem2-3} is unconditional, which allows us to carry out the error analysis without any restriction on the time step.
\end{remark}

\subsection{Extended to higher-order schemes}
With a minor adjustment to system (\ref{the2-3}), we can construct higher-order BDF-$k$ ($k=3, 4$) schemes.

The $k$th order RAV/BDF-$k$ scheme is given by 
\begin{subequations}\label{the2-21}
	\begin{align}
&\frac{\alpha_k\phi^{n+1}-A_k(\phi^{n})}{\Delta t}=\Delta \mu^{n+1}, \label{the2-21a}\\
&\mu^{n+1}=-\epsilon^2\Delta \phi^{n+1}+\lambda \phi^{n+1}+F'(B_k(\overline{\phi}^n))-\lambda B_k(\overline{\phi}^n), \label{the2-21b}\\
&r^{n+1} = 
\begin{cases}  
0, & U_k^{n+1} \geq 0, \\ 
U_k^{n+1}, & U_k^{n+1} < 0, 
\end{cases} \label{the2-21c}\\
&\overline{\phi}^{n+1}=\xi^{n+1}\phi^{n+1}:=\left(\frac{\tilde{E}[\phi^{n+1}]+r^{n+1}}{\tilde{E}[\phi^{n+1}]}\right)\phi^{n+1}, \label{the2-21d}
\end{align}
\end{subequations}
where $\alpha_k$, $U_k^{n+1}$, operators $A_k$ and $B_k$ are as follows:
\begin{equation}\label{the2-22}
\begin{split}
&\alpha_3=\frac{11}{6}, \ A_3(\phi^n)=3\phi^n-\frac{3}{2}\phi^{n-1}+\frac{1}{3}\phi^{n-2}, \ B_3(\overline{\phi}^n)=3\overline{\phi}^n-3\overline{\phi}^{n-1}+\overline{\phi}^{n-2}, \\
& U_3^{n+1}=\frac{18r^n-9r^{n-1}+2r^{n-2}}{11}+\int_\Omega \left(\mu^{n+1}\cdot\frac{11\phi^{n+1}-18\phi^{n}+9\phi^{n-1}-2\phi^{n-2}}{11}\right) \ d\bm{x}\\
&-\left(\frac{11E[\phi^{n+1}]-18E[\phi^{n}]+9E[\phi^{n-1}]-2E[\phi^{n-2}]}{11}\right)+\frac{6}{11}\Delta t\Vert \nabla\mu^{n+1}\Vert_{L^2}^2,
\end{split}
\end{equation}
and
\begin{equation}\label{the2-23}
\begin{split}
&\alpha_4=\frac{25}{12}, \ A_4(\phi^n)=4\phi^n-3\phi^{n-1}+\frac{4}{3}\phi^{n-2}-\frac{1}{4}\phi^{n-3}, \ B_4(\overline{\phi}^n)=4\overline{\phi}^n-6\overline{\phi}^{n-1}+4\overline{\phi}^{n-2}-\overline{\phi}^{n-3},\\
& U_4^{n+1}=\frac{48r^n-36r^{n-1}+16r^{n-2}-3r^{n-3}}{25}+\int_\Omega \left(\mu^{n+1}\cdot\frac{25\phi^{n+1}-48\phi^{n}+36\phi^{n-1}-16\phi^{n-2}+3\phi^{n-3}}{25}\right) \ d\bm{x}\\
&-\left(\frac{25E[\phi^{n+1}]-48E[\phi^{n}]+36E[\phi^{n-1}]-16E[\phi^{n-2}]+3E[\phi^{n-3}]}{25}\right)+\frac{12}{25}\Delta t\Vert \nabla\mu^{n+1}\Vert_{L^2}^2.
\end{split}
\end{equation}

Following the proof of Theorem \ref{theorem2-2}, the energy stability of the RAV/BDF-$k$ scheme can be obtained directly.

\begin{theorem}
\label{theorem2-5}
For the RAV/BDF-$k$ scheme (\ref{the2-21}), it is unconditionally energy stable in the sense that the following discrete energy law holds:
\begin{equation}\label{the2-24}
\begin{split}
\alpha_k(E[\phi^{n+1}]+r^{n+1})-A_k(E[\phi^{n}]+r^n)\leq 0.
\end{split}
\end{equation}
\end{theorem}

\begin{remark}
Unlike the second-order scheme (\ref{the2-3}), the higher-order RAV/BDF-$k$ scheme (\ref{the2-21}) requires incorporating energy $E[\phi]$ into the auxiliary variable $U_k$ in order to achieve the energy stability (\ref{the2-24}). However, in the error analysis, such schemes appear unable to eliminate the time-step constraint. A further discussion of high-order schemes is beyond the scope of this paper and will be dealt with in future work.
\end{remark}

\subsection{RAV approach for gradient flows of multiple functions}
We consider the RAV approach for gradient flows of multiple functions $\phi_1, \phi_2, ..., \phi_k$, with the energy functional:
\begin{equation}\label{the2-25}
\begin{split}
E(\phi_1,\phi_2,...,\phi_k)=\sum\limits_{i=1}^{k}\int_\Omega \frac{1}{2}|\nabla\phi_i|^2 \ d\bm{x} +E_1(\phi_1,\phi_2,...,\phi_k),
\end{split}
\end{equation}
where $E_1=\int_\Omega F(\phi_1,\phi_2,...,\phi_k) \ d\bm{x}$.

We set $L_i=\frac{\delta E_1}{\delta \phi_i}$, $\tilde{E}[\phi_1,...,\phi_k]=E[\phi_1,...,\phi_k]+C_0>0$ and introduce the auxiliary variable $r(t)$ with $r|_{t=0}=0$. The gradient flow equation is given by:
\begin{equation}\label{the2-26}
\begin{split}
&\frac{\partial \phi_i}{\partial t}=\Delta \mu_i,\\
&\mu_i=-\epsilon^2\Delta \phi_i+L_i(\overline{\phi_i}), \\
&\frac{d}{dt}r(t)=\sum\limits_{i=1}^{k}\int_\Omega L_i(\overline{\phi_i})(\phi_i)_t \ d\bm{x}-\frac{d}{dt}E_1,\\
&\overline{\phi_i}=\xi_i\phi_i=\frac{\tilde{E}[\phi]+r}{\tilde{E}[\phi]}\phi_i.
\end{split}
\end{equation}
Based on the method of constructing the gradient flow scheme for a single function, we can easily propose an unconditionally energy stable second-order scheme as follows:
\begin{subequations}\label{the2-27}
	\begin{align}
&\frac{\phi_i^{n+1}-\phi_i^{n}}{\Delta t}=\Delta \mu_i^{n+\frac{1}{2}}, \label{the2-27a}\\
&\mu_i^{n+\frac{1}{2}}=-\epsilon^2\Delta \phi_i^{n+\frac{1}{2}}+\lambda \phi_i^{n+\frac{1}{2}}+L_i((\overline{\phi_i}^*)^{n+\frac{1}{2}})-\lambda(\overline{\phi_i}^*)^{n+\frac{1}{2}}, \label{the2-27b}\\
&r^{n+1} = 
\begin{cases}  
0, & V^{n+1} \geq 0, \\ 
V^{n+1}, & V^{n+1} < 0, 
\end{cases} \label{the2-27c}\\
&\overline{\phi_i}^{n+1}=\xi_i^{n+1}\phi_i^{n+1}:=\left(\frac{\tilde{E}[\phi_1^{n+1},..., \phi_k^{n+1}]+r^{n+1}}{\tilde{E}[\phi_1^{n+1},..., \phi_k^{n+1}]}\right)\phi_i^{n+1}, \label{the2-27d}
\end{align}
\end{subequations}
where 
\begin{equation}\nonumber
\begin{split}
&V^{n+1}=r^n+\sum\limits_{i=1}^{k}\int_\Omega \left(\left(L_i((\overline{\phi_i}^*)^{n+\frac{1}{2}})-\lambda(\overline{\phi_i}^*)^{n+\frac{1}{2}}\right)(\phi_i^{n+1}-\phi_i^n)\right) \ d\bm{x}\\
&-\left(E_1^{n+1}+\frac{\lambda}{2}\sum\limits_{i=1}^{k}\Vert \phi_i^{n+1}\Vert_{L^2}^2-\left(E_1^{n}+\frac{\lambda}{2}\sum\limits_{i=1}^{k}\Vert \phi_i^{n}\Vert_{L^2}^2\right)\right)+\Delta t\sum\limits_{i=1}^{k}\Vert \nabla\mu_i^{n+\frac{1}{2}}\Vert_{L^2}^2.
\end{split}
\end{equation}

Since the nonlinear terms are treated explicitly in the scheme, the variables can be solved sequentially in a simple manner, allowing for a highly efficient implementation. Furthermore, by arguments similar to those used in the proofs of Theorems \ref{theorem2-2} and \ref{theorem2-3}, the stability result for scheme (\ref{the2-21}) can be obtained directly.

\section{Error analysis}
\label{section3}
In this section, we shall derive optimal error estimates for the RAV scheme (\ref{the2-4}) without imposing any restriction on the time step, and the analysis can also be extended to $L^2$ or other types of gradient flows.

We first give some necessary regularity assumptions for the exact solution. 

\begin{assumption}
\label{assumption3-1}
We assume that the exact solution of the system (\ref{the2-1}) satisfies the following regularity condition:
\begin{equation}\label{the3-1}
\begin{split}
\phi \in L^\infty(0,T;W^{1,\infty})\cap W^{1,\infty}(0,T;L^2)\cap W^{2,2}(0,T;H^2)\cap W^{3,2}(0,T;L^2).
\end{split}
\end{equation}
\end{assumption}

We denote that
\begin{equation}\nonumber
\begin{split}
e^n=\phi^n-\phi(t^n), \quad \overline{e}^n=\overline{\phi}^n-\phi(t^n), \quad w^{n}=\mu^n-\mu(t^n).
\end{split}
\end{equation}

The truncation form of the system (\ref{the2-3}) is as follows:
\begin{subequations}\label{the3-2}
\begin{align}
&\frac{\phi(t^{n+1})-\phi(t^{n})}{\Delta t}=\Delta \mu(t^{n+\frac{1}{2}})+R_\phi^{n+\frac{1}{2}}, \label{the3-2a} \\
&\mu(t^{n+\frac{1}{2}})=-\epsilon^2\Delta \phi(t^{n+\frac{1}{2}})+\lambda\phi(t^{n+\frac{1}{2}})+F'(\overline{\phi}^*(t^{n+\frac{1}{2}}))-\lambda\overline{\phi}^*(t^{n+\frac{1}{2}})+R_\mu^{n+\frac{1}{2}}, \label{the3-2b}\\
&\overline{\phi}(t^{n+1})=\phi(t^{n+1}), \label{the3-2c}
\end{align}
\end{subequations}
where $\overline{\phi}^*(t^{n+\frac{1}{2}})=\frac{3}{2}\overline{\phi}(t^n)-\frac{1}{2}\overline{\phi}(t^{n-1})$. Moreover, the truncation errors are given by:
\begin{equation}\label{the3-3}
\begin{split}
R_\phi^{n+\frac{1}{2}}=\frac{1}{2\Delta t}\left(\int_{t^n}^{t^{n+\frac{1}{2}}}(t^n-s)^2\frac{\partial^3\phi}{\partial s^3} \ ds+\int_{t^{n+\frac{1}{2}}}^{t^{n+1}}(t^{n+1}-s)^2\frac{\partial^3\phi}{\partial s^3} \ ds\right),
\end{split}
\end{equation}
and
\begin{equation}\label{the3-4}
\begin{split}
R_\mu^{n+\frac{1}{2}}=F'(\phi(t^{n+\frac{1}{2}}))-F'(\overline{\phi}^*(t^{n+\frac{1}{2}}))-\lambda\left(\phi(t^{n+\frac{1}{2}})-\overline{\phi}^*(t^{n+\frac{1}{2}})\right).
\end{split}
\end{equation}

\begin{theorem}
\label{theorem3-2}
For the $H^{-1}$ gradient flow, we assume that $u^0\in H^3$ and that Assumption \ref{assumption3-1} holds. For the RAV scheme (\ref{the2-4}), there exists a positive constant $C_{T32}$ independent of $\Delta t$ such that
\begin{equation}\label{the3-5}
\begin{split}
\Vert e^{n+1}\Vert_{H^2}^2+\Vert \overline{e}^{n+1}\Vert_{H^2}^2\leq C_{T32}\left(\Delta t\right)^4.
\end{split}
\end{equation}
\end{theorem}

\begin{proof}
From (\ref{the2-4}) and (\ref{the3-2}), we can obtain the error equation for $\phi$ and $\mu$ as
\begin{subequations}\label{the3-6}
\begin{align}
&\frac{e^{n+1}-e^{n}}{\Delta t}=\Delta w^{n+\frac{1}{2}}-R_\phi ^{n+\frac{1}{2}}, \label{the3-6a}\\
&w^{n+\frac{1}{2}}=-\epsilon^2\Delta e^{n+\frac{1}{2}}+\lambda e^{n+\frac{1}{2}}+F'((\overline{\phi}^*)^{n+\frac{1}{2}})-F'(\overline{\phi}^*(t^{n+\frac{1}{2}}))-\lambda\left((\overline{\phi}^*)^{n+\frac{1}{2}}-\overline{\phi}^*(t^{n+\frac{1}{2}})\right) -R_\mu^{n+\frac{1}{2}}. \label{the3-6b}
\end{align}
\end{subequations}

By taking the inner product of (\ref{the3-6a}) and (\ref{the3-6b}) with $\Delta t w^{n+1}$ and $e^{n+1}-e^n$ respectively, we obtain
\begin{equation}\label{the3-7}
\begin{split}
&\frac{\lambda}{2}\left(\Vert e^{n+1}\Vert_{L^2}^2-\Vert e^{n}\Vert_{L^2}^2\right)+\frac{\epsilon^2}{2}\left(\Vert \nabla e^{n+1}\Vert_{L^2}^2-\Vert \nabla e^{n}\Vert_{L^2}^2\right)+\Delta t\Vert \nabla w^{n+\frac{1}{2}}\Vert_{L^2}^2\\
&=J_1+J_2+J_3+J_4,
\end{split}
\end{equation}
where
\begin{equation}\label{the3-8}
\begin{split}
&J_1=\left(F'(\overline{\phi}^*(t^{n+\frac{1}{2}}))-F'((\overline{\phi}^*)^{n+\frac{1}{2}}), e^{n+1}-e^n\right),\\
&J_2=\lambda\left((\overline{\phi}^*)^{n+\frac{1}{2}}-\overline{\phi}^*(t^{n+\frac{1}{2}}), e^{n+1}-e^n\right),\\
&J_3=(R_\mu^{n+\frac{1}{2}}, e^{n+1}-e^n),\\
&J_4=-\Delta t(R_\phi^{n+\frac{1}{2}}, w^{n+1}).
\end{split}
\end{equation}

By the Sobolev embedding theorem $H^2\hookrightarrow L^\infty$ and Theorem \ref{theorem2-3}, we have the following estimates:
\begin{equation}\label{the3-9}
\begin{split}
|J_1|&\leq \Delta t\left|\left(\nabla F'(\overline{\phi}^*(t^{n+\frac{1}{2}}))-\nabla F'((\overline{\phi}^*)^{n+\frac{1}{2}}) , \nabla w^{n+\frac{1}{2}}\right)\right|+\Delta t\left|\left( F'(\overline{\phi}^*(t^{n+\frac{1}{2}}))- F'((\overline{\phi}^*)^{n+\frac{1}{2}}) , R_\phi^{n+\frac{1}{2}} \right)\right|\\
&\leq C_\varepsilon\Delta t\Vert \nabla w^{n+1}\Vert_{L^2}^2+C\Delta t\Vert R_\phi^{n+\frac{1}{2}}\Vert_{L^2}^2+C\Delta t\left\Vert \left(F''(\overline{\phi}^*(t^{n+\frac{1}{2}}))- F''((\overline{\phi}^*)^{n+\frac{1}{2}})\right)\nabla\overline{\phi}^*(t^{n+\frac{1}{2}})\right\Vert_{L^2}^2\\
&+C\Delta t\left\Vert F''((\overline{\phi}^*)^{n+\frac{1}{2}})\left(\nabla\overline{\phi}^*(t^{n+\frac{1}{2}})-\nabla(\overline{\phi}^*)^{n+\frac{1}{2}}\right)\right\Vert_{L^2}^2+C\Delta t\Vert F'(\overline{\phi}^*(t^{n+\frac{1}{2}}))- F'((\overline{\phi}^*)^{n+\frac{1}{2}})\Vert_{L^2}^2\\
&\leq C\Delta t\left(\Vert \overline{e}^{n}\Vert_{L^2}^2+\Vert \nabla\overline{e}^{n}\Vert_{L^2}^2+\Vert \overline{e}^{n-1}\Vert_{L^2}^2+\Vert \nabla\overline{e}^{n-1}\Vert_{L^2}^2\right)\\
&+C_\varepsilon\Delta t\Vert \nabla w^{n+1}\Vert_{L^2}^2+C\left(\Delta t\right)^4\int_{t^n}^{t^{n+1}}\left\Vert \frac{\partial^3 \phi}{\partial s^3}\right\Vert_{L^2}^2 \ ds.
\end{split}
\end{equation}
Similarly
\begin{equation}\label{the3-10}
\begin{split}
|J_2|&\leq C\Delta t\left(\Vert \overline{e}^{n}\Vert_{L^2}^2+\Vert \overline{e}^{n-1}\Vert_{L^2}^2\right)+C_\varepsilon\Delta t\Vert \nabla w^{n+1}\Vert_{L^2}^2+C\left(\Delta t\right)^4\int_{t^n}^{t^{n+1}}\left\Vert \frac{\partial^3 \phi}{\partial s^3}\right\Vert_{L^2}^2 \ ds.
\end{split}
\end{equation}

Note that 
\begin{equation}\label{the3-11}
\begin{split}
\Vert \nabla R_\mu^{n+\frac{1}{2}}\Vert_{L^2}^2 &\leq C\Vert F''(\phi(t^{n+\frac{1}{2}}))\nabla\phi(t^{n+\frac{1}{2}})-F''(\overline{\phi}^*(t^{n+\frac{1}{2}}))\nabla\overline{\phi}^*(t^{n+\frac{1}{2}})\Vert_{L^2}^2+C\Vert \nabla \phi(t^{n+\frac{1}{2}})-\nabla\overline{\phi}^*(t^{n+\frac{1}{2}})\Vert_{L^2}^2\\
&\leq C(\Delta t)^3\int_{t^n}^{t^{n+1}}\left\Vert \frac{\partial^2 \phi}{\partial s^2}\right\Vert_{H^{1}}^2 \ ds.
\end{split}
\end{equation}
For $J_3$ and $J_4$, we estimate 
\begin{equation}\label{the3-12}
\begin{split}
|J_3|&\leq C_\varepsilon\Delta t\Vert \nabla w^{n+1}\Vert_{L^2}^2+C\left(\Delta t\right)^4\int_{t^n}^{t^{n+1}}\left\Vert \frac{\partial^3 \phi}{\partial s^3}\right\Vert_{L^2}^2 \ ds+C\Delta t\Vert \nabla R_\mu^{n+\frac{1}{2}}\Vert_{L^2}^2\\
&\leq C_\varepsilon\Delta t\Vert \nabla w^{n+1}\Vert_{L^2}^2+C\left(\Delta t\right)^4\int_{t^n}^{t^{n+1}}\left\Vert \frac{\partial^3 \phi}{\partial s^3}\right\Vert_{L^2}^2 \ ds+C(\Delta t)^4\int_{t^n}^{t^{n+1}}\left\Vert \frac{\partial^2 \phi}{\partial s^2}\right\Vert_{H^{1}}^2 \ ds,
\end{split}
\end{equation}
and 
\begin{equation}\label{the3-13}
\begin{split}
|J_4|&\leq C_\varepsilon\Delta t\Vert \nabla w^{n+1}\Vert_{L^2}^2+C\Delta t\Vert \nabla R_\mu^{n+\frac{1}{2}}\Vert_{L^2}^2\\
&\leq C_\varepsilon\Delta t\Vert \nabla w^{n+1}\Vert_{L^2}^2+C(\Delta t)^4\int_{t^n}^{t^{n+1}}\left\Vert \frac{\partial^2 \phi}{\partial s^2}\right\Vert_{H^{1}}^2 \ ds.
\end{split}
\end{equation}

By (\ref{the2-4d}), we notice that
\begin{equation}\label{the3-14}
\begin{split}
\overline{e}^{n}=(\xi^n-1)\phi^n+e^{n}=\frac{r^{n}}{\tilde{E}[\phi^{n}]}\phi^n+e^{n}.
\end{split}
\end{equation}
It follows from Theorem \ref{theorem2-3} that
\begin{equation}\label{the3-15}
\begin{split}
&\Vert\overline{e}^{n}\Vert_{L^2}^2\leq \Vert e^{n}\Vert_{L^2}^2+C|r^n|^2\Vert\phi^n\Vert_{L^2}^2\leq  \Vert e^{n}\Vert_{L^2}^2+C|r^n|^2,\\
&\Vert\nabla \overline{e}^{n}\Vert_{L^2}^2\leq \Vert \nabla e^{n}\Vert_{L^2}^2+C|r^n|^2\Vert \nabla \phi^n\Vert_{L^2}^2\leq \Vert \nabla e^{n}\Vert_{L^2}^2+C|r^n|^2.
\end{split}
\end{equation}

Combining (\ref{the3-7})-(\ref{the3-15}), we have
\begin{equation}\label{the3-16}
\begin{split}
&\left(\Vert e^{n+1}\Vert_{L^2}^2-\Vert e^{n}\Vert_{L^2}^2\right)+\left(\Vert \nabla e^{n+1}\Vert_{L^2}^2-\Vert \nabla e^{n}\Vert_{L^2}^2\right)+\Delta t\Vert \nabla w^{n+\frac{1}{2}}\Vert_{L^2}^2\\
&\leq C\Delta t\left(\Vert {e}^{n}\Vert_{L^2}^2+\Vert \nabla{e}^{n}\Vert_{L^2}^2+\Vert {e}^{n-1}\Vert_{L^2}^2+\Vert \nabla{e}^{n-1}\Vert_{L^2}^2\right)+C\Delta t(|r^n|^2+|r^{n-1}|^2)\\
&+C(\Delta t)^4\left(\int_{t^n}^{t^{n+1}}\left\Vert \frac{\partial^3 \phi}{\partial s^3}\right\Vert_{L^2}^2 \ ds+\int_{t^n}^{t^{n+1}}\left\Vert \frac{\partial^2 \phi}{\partial s^2}\right\Vert_{H^{1}}^2 \ ds\right).
\end{split}
\end{equation}

Summing over $n$ from $1$ to $m\leq \frac{T}{\Delta t}-1$ and using the discrete Gronwall inequality (see \cite{quarteroni1994numerical}), we obtain
\begin{equation}\label{the3-17}
\begin{split}
&\Vert e^{m+1}\Vert_{L^2}^2+\Vert \nabla e^{m+1}\Vert_{L^2}^2+\Delta t\sum\limits_{i=1}^{m}\Vert \nabla w^{i+\frac{1}{2}}\Vert_{L^2}^2\leq C(\Delta t)^4+C\Delta t\sum\limits_{i=1}^{m}\left(|r^i|^2+|r^{i-1}|^2\right).
\end{split}
\end{equation}

Since $r^{n+1}=0$ for $Q^{n+1}\geq 0$, we only need to consider the case $r^{n+1}=Q^{n+1}< 0$. By the mean value theorem, (\ref{the2-4c}) can be rewritten as
\begin{equation}\label{the3-18}
\begin{split}
&r^{n+1}-r^{n}-\Delta t\Vert \nabla\mu^{n+\frac{1}{2}}\Vert_{L^2}^2=\int_\Omega \left(\left(F'((\overline{\phi}^*)^{n+\frac{1}{2}})-\lambda(\overline{\phi}^*)^{n+\frac{1}{2}}\right)(\phi^{n+1}-\phi^n)\right) \ d\bm{x}\\
&-\int_\Omega \left(F'(\phi^n+\theta^n(\phi^{n+1}-\phi^{n}))-\lambda(\phi^n+\theta^n(\phi^{n+1}-\phi^{n}))\right) (\phi^{n+1}-\phi^{n}) \ d\bm{x},
\end{split}
\end{equation}
where $\theta^n\in (0, 1)$. Let $G(\phi)=F(\phi)-\frac{\lambda}{2}\phi^2$. The corresponding error equation is given by:
\begin{equation}\label{the3-19}
\begin{split}
&|r^{n+1}|-|r^{n}|+\Delta t\Vert \nabla\mu^{n+\frac{1}{2}}\Vert_{L^2}^2=\int_\Omega G'\left(\phi^n+\theta^n(\phi^{n+1}-\phi^{n})\right)\left(e^{n+1}-e^n\right) \ d\bm{x}\\
&+\int_\Omega  \left(G'\left(\phi^n+\theta^n(\phi^{n+1}-\phi^{n})\right)-G'\left(\phi(t^n)+\theta^n(\phi(t^{n+1})-\phi(t^{n}))\right)\right)(\phi(t^{n+1})-\phi(t^{n})) \ d\bm{x}\\
&-\int_\Omega \left(G'((\overline{\phi}^*)^{n+\frac{1}{2}})-G'(\overline{\phi}^*(t^{n+\frac{1}{2}}))\right)(\phi(t^{n+1})-\phi(t^{n})) \ d\bm{x}\\
&-\int_\Omega G'((\overline{\phi}^*)^{n+\frac{1}{2}})(e^{n+1}-e^n)\ d\bm{x}-R_\theta^{n}\\
&=:K_1+K_2+K_3+K_4+R_\theta^{n},
\end{split}
\end{equation}
where $R_\theta^n=\int_\Omega \left(G'\left(\phi(t^n)+\theta^n(\phi(t^{n+1})-\phi(t^{n}))\right)-G'(\overline{\phi}^*(t^{n+\frac{1}{2}}))\right)(\phi(t^{n+1})-\phi(t^{n})) \ d\bm{x}$.

By using Theorem \ref{theorem2-3}, we have
\begin{equation}\label{the3-20}
\begin{split}
&|K_1|+|K_4|\leq C\Delta t\left(\Vert \nabla w^{n+1}\Vert_{L^2}+\Vert R_\phi^{n+\frac{1}{2}}\Vert_{L^2}\right),\\
&|K_2|+|K_3|\leq C\Delta t\left(\Vert e^{n+1}\Vert_{L^2}+\Vert e^{n}\Vert_{L^2}+\Vert e^{n-1}\Vert_{L^2}+|r^n|+|r^{n-1}|\right),\\
&|R_\theta^n|\leq C\left(\Delta t\right)^2.
 \end{split}
\end{equation}

From (\ref{the3-17}), (\ref{the3-19}) and (\ref{the3-20}), and by using the Cauchy-Schwarz inequality, we obtain 
\begin{equation}\label{the3-21}
\begin{split}
|r^{n+1}|^2\leq C(\Delta t)^2+C\Delta t\sum\limits_{i=1}^{n}\left(|r^i|^2+|r^{i-1}|^2\right).
\end{split}
\end{equation}

By applying the discrete Gronwall inequality, it is easy to obtain the first-order estimate
\begin{equation}\label{the3-22}
\begin{split}
\Vert e^{n+1}\Vert_{H^1}^2+|r^{n+1}|^2\leq C(\Delta t)^2.
\end{split}
\end{equation}

We next derive the second-order estimate. Consider the error equation for $r$ at $t^{n+\frac{1}{2}}$:
\begin{equation}\label{the3-23}
\begin{split}
&|r^{n+1}|-|r^{n}|+\Delta t\Vert \nabla\mu^{n+\frac{1}{2}}\Vert_{L^2}^2=\int_\Omega \left(G\left(\phi(t^{n+1})+\zeta^{n+1}e^{n+1}\right)-G\left(\phi(t^{n})+\zeta^n e^{n}\right)\right)e^{n+1} \ d\bm{x}\\
&\int_\Omega  G\left(\phi(t^{n})+\zeta^n e^{n}\right)(e^{n+1}-e^n) \ d\bm{x}+K_3+K_4+R_r^{n+\frac{1}{2}},
\end{split}
\end{equation}
where $\zeta^n$ is between $\phi(t^n)$ and $\phi^n$. The truncation error
\begin{equation}\label{the3-24}
\begin{split}
&R_r^{n+\frac{1}{2}}=\Delta t \int_\Omega \left(G'(\phi({t^{n+\frac{1}{2}}})-G'(\overline{\phi}^*(t^{n+\frac{1}{2}}))\right)\phi_t(t^{n+\frac{1}{2}}) \ d\bm{x}\\
&+\int_\Omega G'(\overline{\phi}^*(t^{n+\frac{1}{2}}))\left(\Delta t\phi_t(t^{n+\frac{1}{2}})-(\phi(t^{n+1})-\phi(t^{n}))\right) \ d\bm{x}\\
&+\frac{1}{2}\int_\Omega \left(\int_{t^n}^{t^{n+\frac{1}{2}}}(t^n-s)^2\left(G''(\frac{\partial \phi}{\partial s})^3+2G'\frac{\partial^2 \phi}{\partial s^2}+G'\frac{\partial \phi}{\partial s}\frac{\partial^2 \phi}{\partial s^2}+G\frac{\partial^3 \phi}{\partial s^3}\right) \ ds\right)           \ d\bm{x} \\
&+\frac{1}{2}\int_\Omega \left(\int_{t^{n+\frac{1}{2}}}^{t^{n+1}}(t^{n+1}-s)^2 \left(G''(\frac{\partial \phi}{\partial s})^3+2G'\frac{\partial^2 \phi}{\partial s^2}+G'\frac{\partial \phi}{\partial s}\frac{\partial^2 \phi}{\partial s^2}+G\frac{\partial^3 \phi}{\partial s^3}\right) \ ds\right)           \ d\bm{x}.
\end{split}
\end{equation}
It follows that
\begin{equation}\label{the3-25}
\begin{split}
&|R_r^{n+\frac{1}{2}}|\leq C\Delta t\left|\int_{t^n}^{t^{n+1}}(t^n-s)\left\Vert\frac{\partial^2 \phi}{\partial s^2}\right\Vert_{L^2} \ ds\right|+C\int_{t^n}^{t^{n+1}}(t^n-s)^2\left(1+\left\Vert\frac{\partial^3\phi}{\partial s^3}\right\Vert_{L^2}+\left\Vert\frac{\partial^2\phi}{\partial s^2}\right\Vert_{L^2}\right) \ ds.
\end{split}
\end{equation}

By using (\ref{the3-22}), the first and second terms on the right-hand side of (\ref{the3-23}) can be estimated by:
\begin{equation}\label{the3-26}
\begin{split}
&\left|\int_\Omega \left(G\left(\phi(t^{n+1})+\zeta^{n+1}e^{n+1}\right)-G\left(\phi(t^{n})+\zeta^n e^{n}\right)\right)e^{n+1} \ d\bm{x}\right|\leq \Delta t\Vert e^{n+1}\Vert_{L^2}+\Vert e^{n+1}\Vert_{L^2}^2+\Vert e^{n}\Vert_{L^2}^2\\
&\leq C\Delta t\left(\Vert e^{n+1}\Vert_{L^2}+\Vert e^{n}\Vert_{L^2}\right),\\
&\left|\int_\Omega  G\left(\phi(t^{n})+\zeta^n e^{n}\right)(e^{n+1}-e^n) \ d\bm{x}\right|\leq C\Delta t\left(\Vert \nabla w^{n+1}\Vert_{L^2}+\Vert R_\phi^{n+\frac{1}{2}}\Vert_{L^2}\right).
\end{split}
\end{equation}

Combining (\ref{the3-23})-(\ref{the3-26}), we arrive at
\begin{equation}\label{the3-27}
\begin{split}
|r^{n+1}|^2&\leq C(\Delta t)^4+C\Delta t\sum\limits_{i=1}^{n}\left(|r^i|^2+|r^{i-1}|^2\right).
\end{split}
\end{equation}
By using the Gronwall inequality again, we obtain 
\begin{equation}\label{the3-28}
\begin{split}
\Vert e^{n+1}\Vert_{H^1}^2+\Vert \overline{e}^{n+1}\Vert_{H^1}^2+|r^{n+1}|^2\leq C(\Delta t)^4.
\end{split}
\end{equation}

Finally, we derive the estimate for $\Vert e^{n+1}\Vert_{H^2}$. Combining (\ref{the3-6a}) with (\ref{the3-6b}) and taking the inner product with $\Delta^2 e^{n+\frac{1}{2}}$ gives
\begin{equation}\label{the3-29}
\begin{split}
&\frac{1}{2}\left(\Vert \Delta e^{n+1}\Vert_{L^2}^2-\Vert \Delta e^{n}\Vert_{L^2}^2\right)+\lambda\Delta t\Vert \nabla\Delta e^{n+\frac{1}{2}}\Vert_{L^2}^2+\epsilon\Delta t\Vert \Delta^2 e^{n+\frac{1}{2}}\Vert_{L^2}^2\\
&=\Delta t\left(\Delta G'((\overline{\phi}^*)^{n+\frac{1}{2}})-\Delta G'(\overline{\phi}^*(t^{n+\frac{1}{2}})), \Delta^2 e^{n+\frac{1}{2}}\right)-\Delta t\left(R_\phi^{n+\frac{1}{2}}+\Delta R_\mu^{n+\frac{1}{2}}, \Delta^2 e^{n+\frac{1}{2}}\right).
\end{split}
\end{equation}
Note that 
\begin{equation}\label{the3-30}
\begin{split}
\Delta G'(\phi)=G'(\phi)\Delta \phi+G''(\phi)|\nabla\phi|^2,
\end{split}
\end{equation}
and 
\begin{equation}\label{the3-31}
\begin{split}
&\Vert\Delta G'((\overline{\phi}^*)^{n+\frac{1}{2}})-\Delta G'(\overline{\phi}^*(t^{n+\frac{1}{2}}))\Vert_{L^2}\\
&\leq C\left(\Vert G'((\overline{\phi}^*)^{n+\frac{1}{2}})\Delta\left(\overline{e}^n-\overline{e}^{n-1}\right)\Vert_{L^2}+\Vert (G'((\overline{\phi}^*)^{n+\frac{1}{2}})-G'(\overline{\phi}^*(t^{n+\frac{1}{2}}))) \Delta \overline{\phi}^*(t^{n+\frac{1}{2}})\Vert_{L^2}\right) \\
&+C\left(\left\Vert  G'((\overline{\phi}^*)^{n+\frac{1}{2}})\left(|\nabla (\overline{\phi}^*)^{n+\frac{1}{2}}|^2-|\nabla \overline{\phi}^*(t^{n+\frac{1}{2}})|^2\right)\right\Vert_{L^2}+\Vert (G''((\overline{\phi}^*)^{n+\frac{1}{2}})-G''(\overline{\phi}^*(t^{n+\frac{1}{2}})))\Vert_{L^2}\right)\\
&\leq C\left(\Vert \overline{e}^{n}\Vert_{L^2}+\Vert \overline{e}^{n-1}\Vert_{L^2}+\Vert \nabla \overline{e}^{n}\Vert_{L^2}+\Vert \nabla \overline{e}^{n-1}\Vert_{L^2}+\Vert \Delta \overline{e}^{n}\Vert_{L^2}+\Vert \Delta \overline{e}^{n-1}\Vert_{L^2}\right),
\end{split}
\end{equation}
where we used the Sobolev embedding $H^2\hookrightarrow W^{1,4}$.

Therefore, the term on the right-hand side of (\ref{the3-29}) can be estimated as
\begin{equation}\label{the3-32}
\begin{split}
&\left|\Delta t\left(\Delta G'((\overline{\phi}^*)^{n+\frac{1}{2}})-\Delta G'(\overline{\phi}^*(t^{n+\frac{1}{2}})), \Delta^2 e^{n+\frac{1}{2}}\right)-\Delta t\left(R_\phi^{n+\frac{1}{2}}+\Delta R_\mu^{n+\frac{1}{2}}, \Delta^2 e^{n+\frac{1}{2}}\right)\right|\\
&\leq C(\Delta t)^5+C_{\varepsilon}\Delta t\Vert \Delta^2 e^{n+\frac{1}{2}}\Vert_{L^2}^2+C\Delta t\left(\Vert \Delta {e}^{n}\Vert_{L^2}^2+\Vert \Delta {e}^{n-1}\Vert_{L^2}^2+|r^n|^2+|r^{n-1}|^2\right)\\
&+\Delta t\Vert R_\phi^{n+\frac{1}{2}}\Vert_{L^2}^2+\Delta t\Vert \Delta R_\mu^{n+\frac{1}{2}}\Vert_{L^2}^2.
\end{split}
\end{equation}
Then combining (\ref{the3-29}) with (\ref{the3-28}) and (\ref{the3-32}) leads to
\begin{equation}\label{the3-33}
\begin{split}
\Vert \Delta e^{n+1}\Vert_{L^2}\leq C(\Delta t)^4(1+\Vert \phi\Vert_{W^{3,2}(0,T;L^2)}^2+\Vert \phi\Vert_{W^{2,2}(0,T;H^2)}^2)+C\Delta t\sum\limits_{i=1}^{n}\left(\Vert \Delta {e}^{n}\Vert_{L^2}^2+\Vert \Delta {e}^{n-1}\Vert_{L^2}^2\right).
\end{split}
\end{equation}

By applying the discrete Gronwall inequality, we can obtain the desired result (\ref{the3-5}).
\end{proof}

\section{Numerical experiments}
\label{section4}
In this section, we first test the proposed RAV scheme on the classical Cahn–Hilliard and phase-field crystal models to demonstrate its accuracy, stability, and efficiency. Moreover, we compare the obtained results with those of the original SAV method. Subsequently, we apply the scheme to the challenging phase-field vesicle model and the surfactant model. In all examples, we consider periodic or homogeneous Neumann boundary conditions, and the finite element method is used for spatial discretization.

In Subsections \ref{subsection4-1} and \ref{subsection4-2}, we compare the second-order RAV scheme (\ref{the2-3}) with the following SAV-CN scheme:
\begin{equation}\label{the4-1}
\begin{split}
&\frac{\phi^{n+1}-\phi^n}{\Delta t}=\mathcal{G}\mu^{n+\frac{1}{2}}, \\
&\mu^{n+\frac{1}{2}}=\mathcal{L}{\phi^{n+\frac{1}{2}}}+\frac{r^{n+1}+r^n}{2\sqrt{\int_\Omega F((\overline{\phi}^*)^{n+\frac{1}{2}}) \ d\bm{x}+C_0}}F'((\overline{\phi}^*)^{n+\frac{1}{2}}),\\
&\frac{r^{n+1}-r^n}{\Delta t}=\frac{1}{2\sqrt{\int_\Omega F((\overline{\phi}^*)^{n+\frac{1}{2}}) d\bm{x}+C_0}}\int_\Omega F'((\overline{\phi}^*)^{n+\frac{1}{2}}) \frac{\phi^{n+1}-\phi^{n}}{\Delta t}\ d\bm{x}.
\end{split}
\end{equation}
For both the SAV and RAV schemes, the stabilization parameter is set to $\lambda=2$, unless specified otherwise.

\subsection{Cahn-Hilliard model}
\label{subsection4-1}
The Cahn-Hilliard equation \cite{cahn1958free} is a fundamental diffuse-interface model that captures phase separation dynamics. The corresponding free energy density is $F(\phi)=\frac{1}{4}(\phi^2-1)^2$, and the computational domain is set to $\Omega=[0,2\pi]^2$. The periodic boundary conditions and the initial condition $\phi_0(x,y)=0.05\sin(x)\sin(y)$ are imposed.

We first test the temporal convergence of the second‑order RAV scheme(\ref{the2-3}) on a sufficiently fine spatial mesh. A finer time step $\Delta t^r=1e-5$ is used to compute a reference solution up to $T=0.016$. The computational error, convergence rates and auxiliary variable $r$ are shown in Table \ref{table1}. It can be directly seen that the proposed RAV scheme achieves the desired temporal accuracy, and the auxiliary variable $r$ remains fully consistent with the original variable $0$ for all tested small time steps.

\begin{table}[H]
\caption{Errors and convergence rates for $\phi$ of the second-order RAV scheme for the Cahn–Hilliard equation at $T=0.016$.}\label{table1}
\centering
\scalebox{0.65}{
\resizebox{\linewidth}{!}{
\begin{tabular}{c c c c c c} \hline  		
 $\Delta t$   &  $L^2$-error  &  Order    &  $L^\infty$-error  &  Order  & $\max_{i\leq n} r^i$  \\ \hline
 $1.6e-3$		&  1.057e-5  &  &  4.417e-5 &  &  0.0 \\ [6pt]
 $8e-4$	&  2.726e-6 &  1.96 &  1.134e-5    & 1.96 &  0.0 \\  [6pt]
$4e-4$	&  6.867e-7  &   1.99  &  2.638e-6 &  2.10 &  0.0 \\  [6pt]
$2e-4$	&  1.648e-7   &  2.06 &   5.799e-7   &  2.18 &  0.0 \\  [6pt]
$1e-4$	&  3.996e-8   &  2.04 &   1.394e-7    &  2.05 & 0.0 \\  \hline 
\end{tabular}}
	}
\end{table}

We set the final time to $T=5$ and use relatively large time steps to further test the robustness and accuracy of the SAV and RAV methods under coarse temporal resolution. As shown in Figures \ref{figure1} and \ref{figure2}, the snapshots in panels (a)–(c) clearly demonstrate that the proposed RAV scheme exhibits significantly reduced numerical dissipation, resulting in a more stable and physically reliable solution compared with the SAV method. Furthermore, panel (d) shows the error between the auxiliary variable and the original variable. For the SAV method, the error increases over time, leading to a noticeable loss of accuracy in the numerical solution. In comparison, the RAV method preserves a consistently small discrepancy, demonstrating its superior accuracy.

\begin{figure}[H]
	\centering
	\subfigure[$\Delta t=1/2$]{
		\includegraphics[scale=0.16]{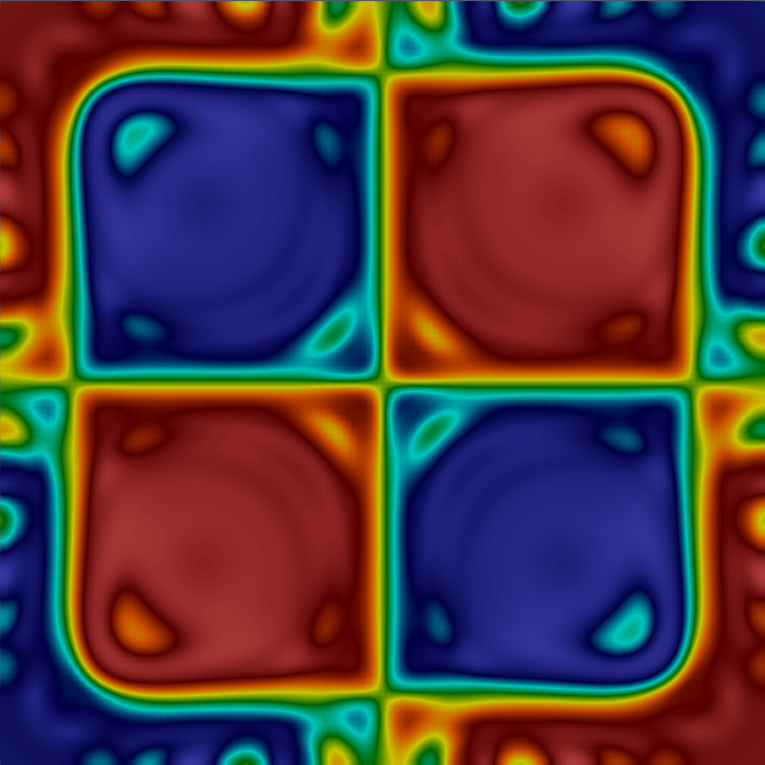}
	}
	\subfigure[$\Delta t=1/4$]{
		\includegraphics[scale=0.16]{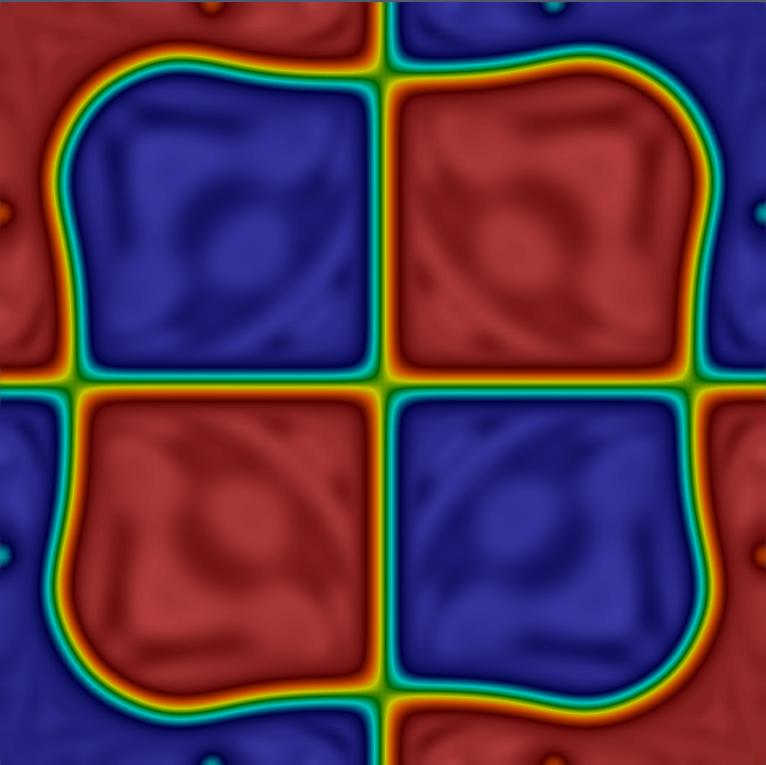}
	}
	\subfigure[$\Delta t=1/8$]{
		\includegraphics[scale=0.16]{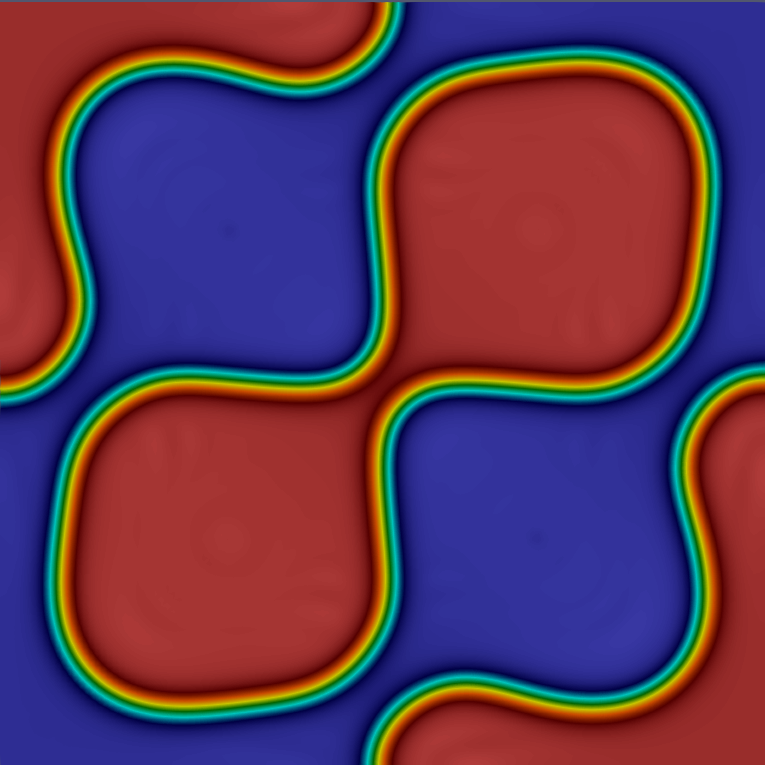}
	}
	\subfigure[$|r_{\text{sav}}-\sqrt{\int_\Omega F(\phi)d\bm{x}+C_0}|$]{
		\includegraphics[scale=0.16]{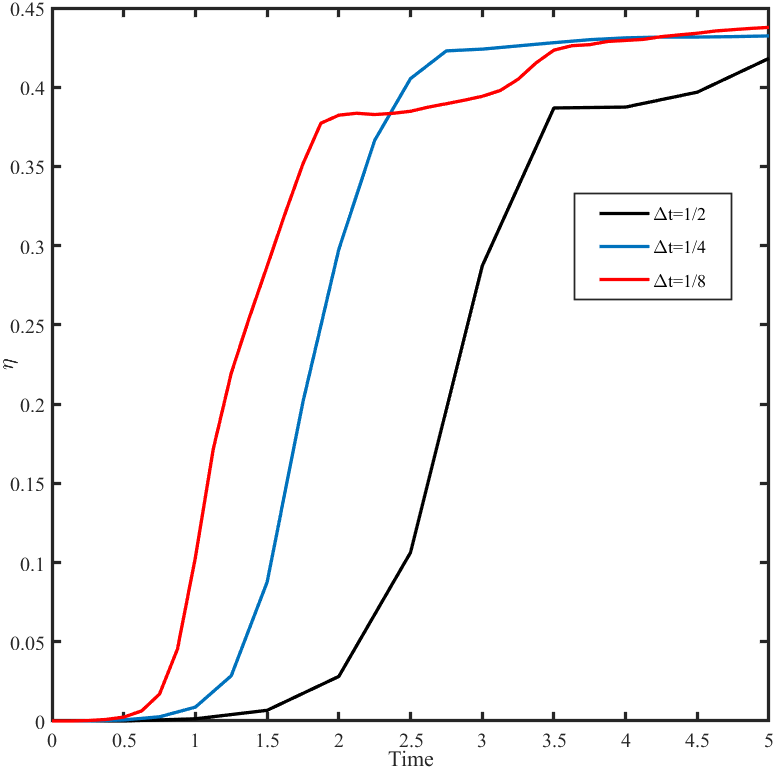}
	}
	
	\caption{Snapshots of the phase variable $\phi$ computed by the SAV–CN scheme at $t=5$. The line graphs give the discrepancy between the auxiliary variable and the original variable.}
	\label{figure1}
\end{figure}

\begin{figure}[H]
	\centering
	\subfigure[$\Delta t=1/2$]{
		\includegraphics[scale=0.16]{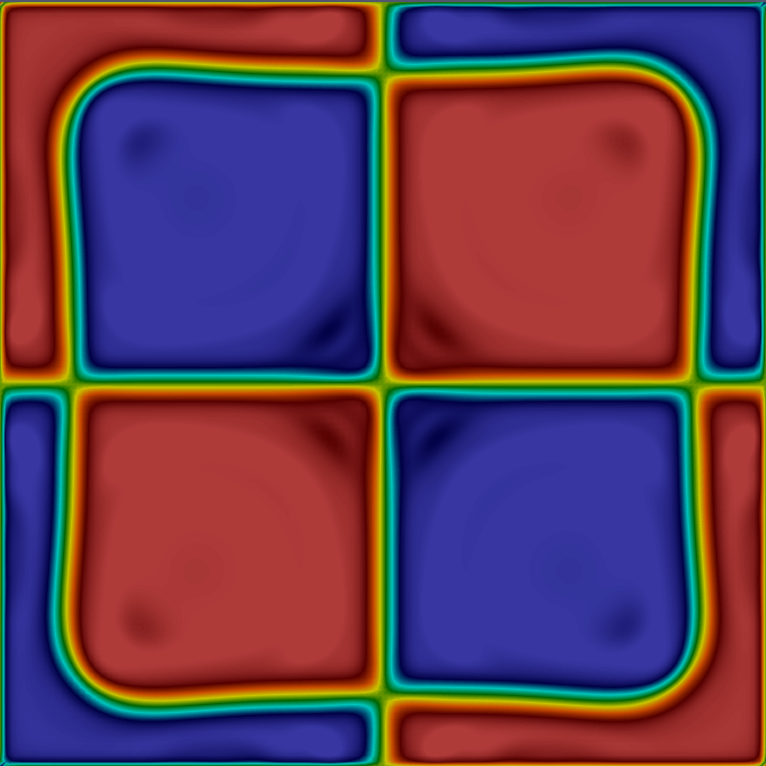}
	}
	\subfigure[$\Delta t=1/4$]{
		\includegraphics[scale=0.16]{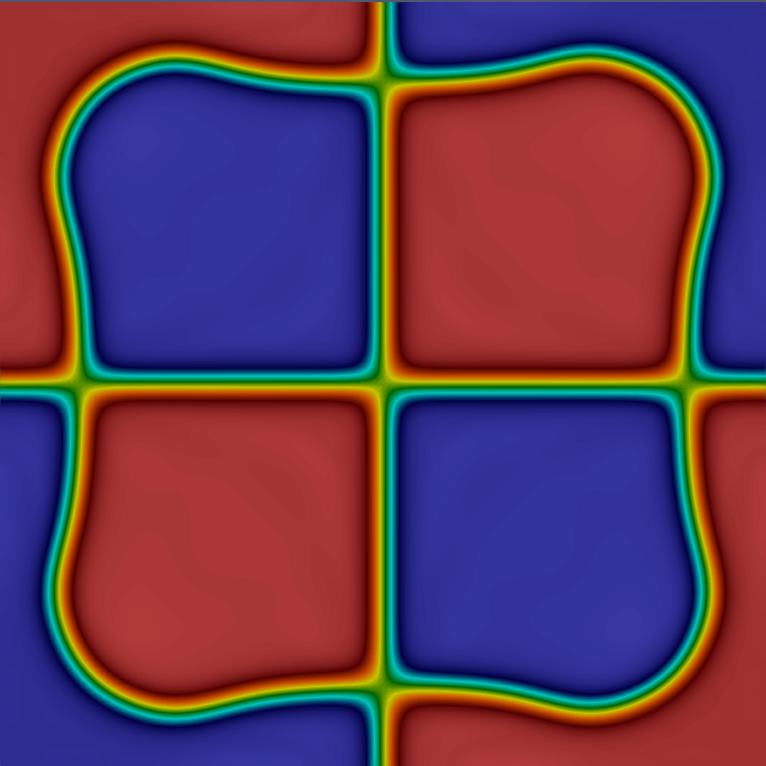}
	}
	\subfigure[$\Delta t=1/8$]{
		\includegraphics[scale=0.16]{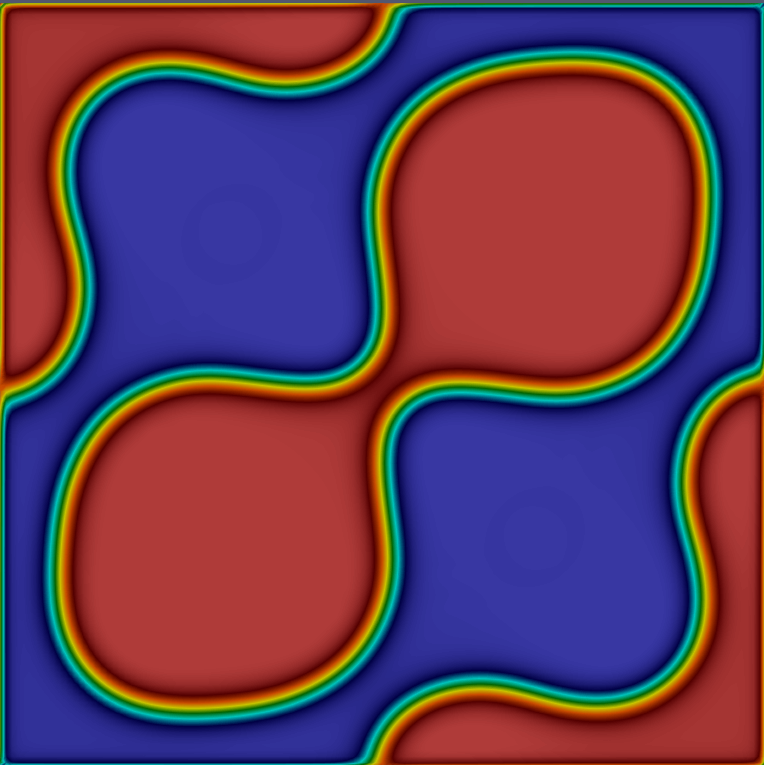}
	}
	\subfigure[$r_{\text{rav}}-0$]{
		\includegraphics[scale=0.16]{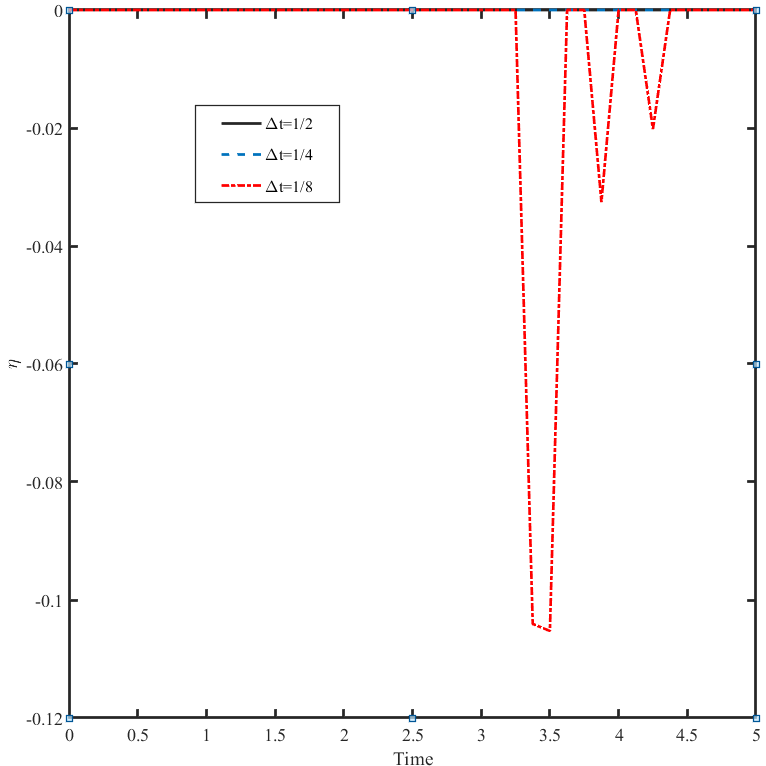}
	}
	
	\caption{Snapshots of the phase variable $\phi$ computed by the RAV scheme (\ref{the2-3}) at $t=5$. The line graphs give the discrepancy between the auxiliary variable and the original variable.}
	\label{figure2}
\end{figure}

Figure \ref{figure3} shows the time evolution of $Q^{n+1}-r^n$ and the total energy. For the cases $\Delta t=1/2$ and $\Delta t=1/4$, it can be observed that $Q^{n+1}-r^n\geq 0$, which corresponds to a decrease in both the original energy and the modified energy. It is worth noting that for $\Delta t=1/8$, there are certain moments at which $Q^{n+1}-r^n< 0$, and at these moments the scheme fails to satisfy the original energy law, whereas the modified energy law still does. These numerical results verify the correctness of Theorem \ref{theorem2-2} regarding its characterization of the relationship between the original and modified energies.

\begin{figure}[H]
	\centering
	\subfigure[$(Q^{n+1}-r^n)$ vs. time]{
		\includegraphics[scale=0.38]{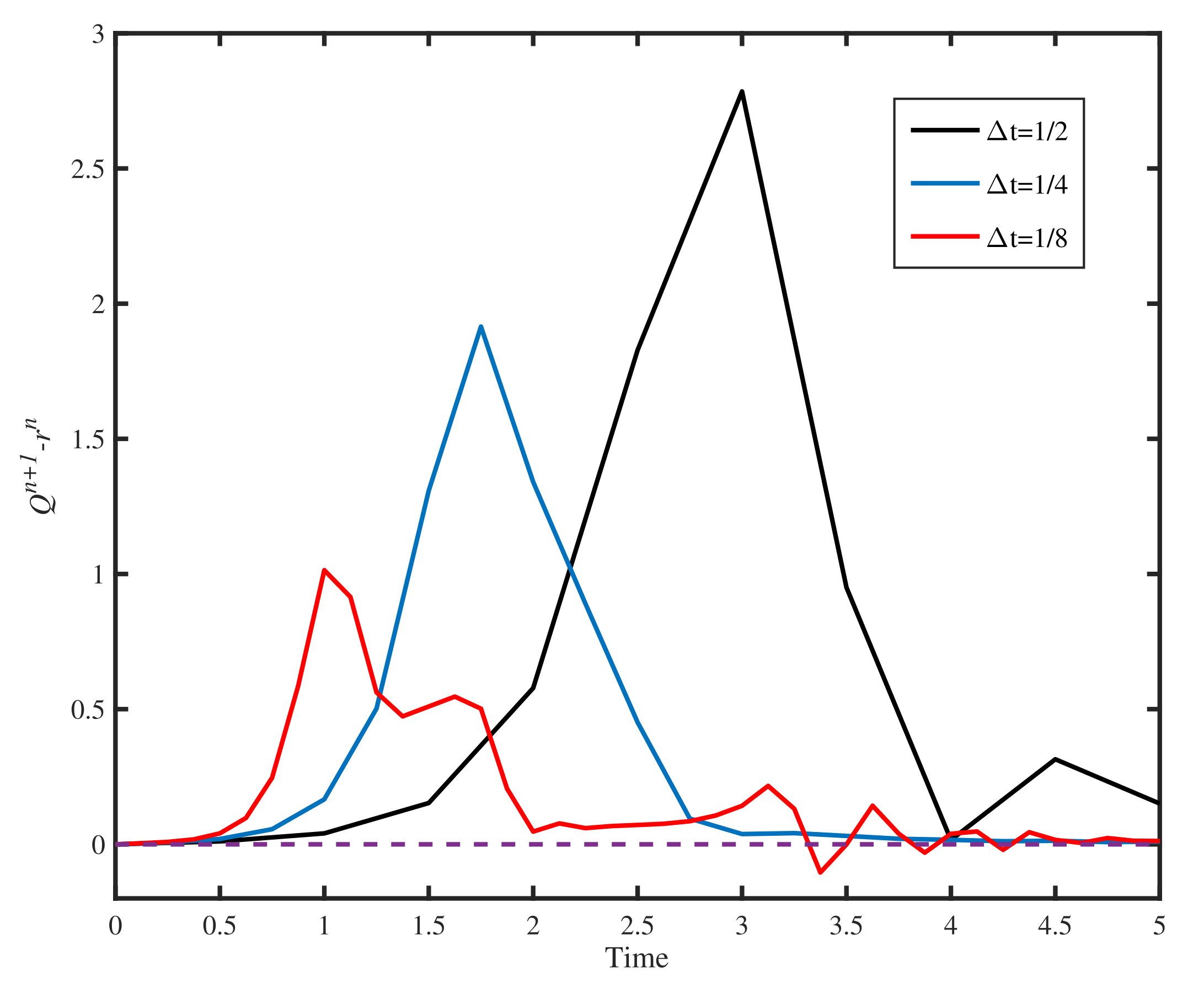}
	}
	\subfigure[Energy vs. time]{
		\includegraphics[scale=0.38]{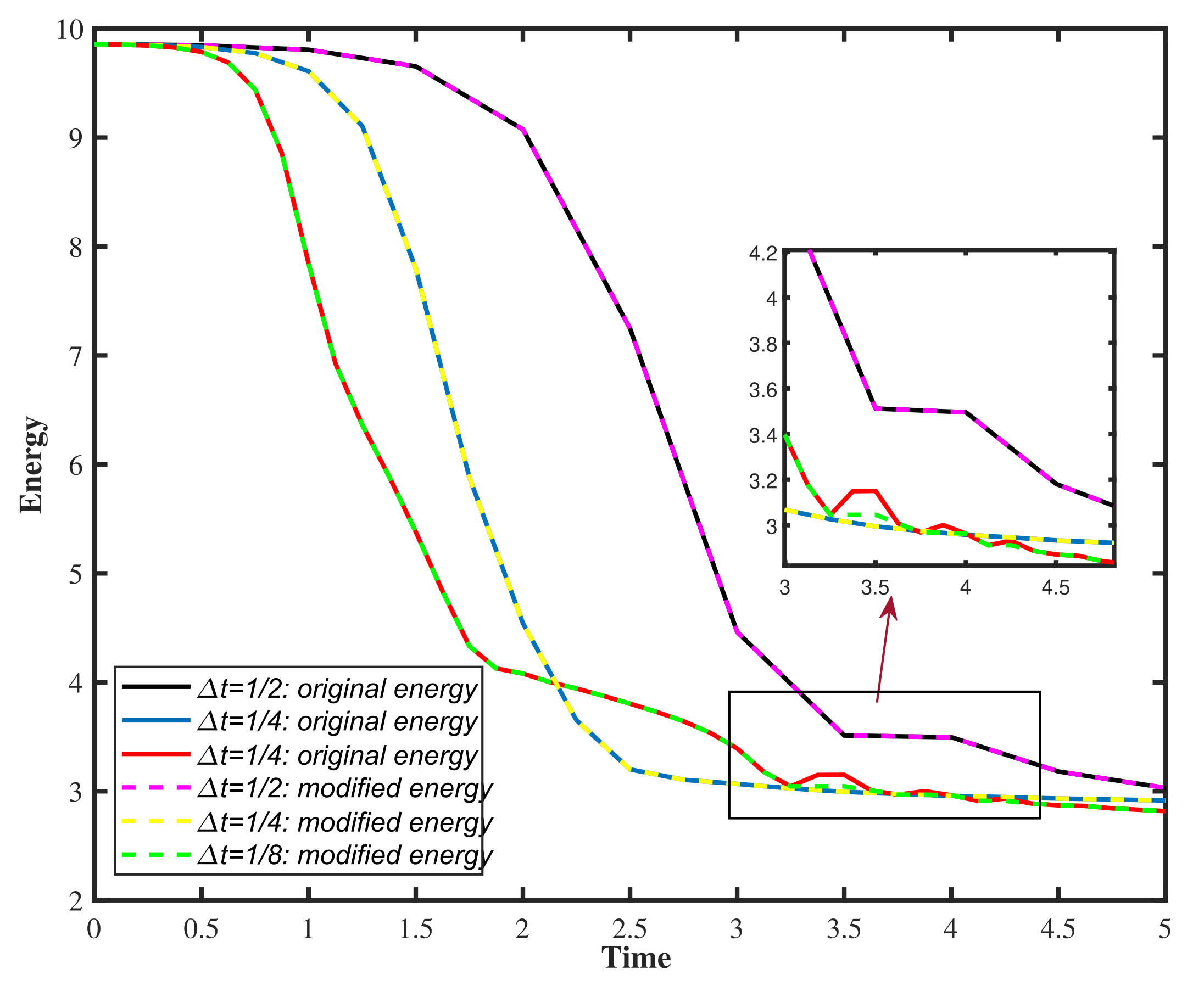}
	}
	\caption{Evolution of total energy for the RAV scheme with different time steps.}
	\label{figure3}
\end{figure}

\subsection{Phase-field crystal (PFC) model}
\label{subsection4-2}
The phase-field crystal model \cite{elder2002modeling} is capable of resolving diffusive time scales while retaining structural information at the atomic level, and has therefore become a powerful tool for investigating microstructure evolution in a wide range of materials systems.

Herein we consider the Swift-Hohenberg free energy \cite{swift1977hydrodynamic}:
\begin{equation}\label{the4-2}
\begin{split}
E[\phi]=\int_{\Omega}\left(\frac{1}{2}\phi(1+\Delta)^2\phi+\frac{1}{4}\phi^4-\frac{1}{3}\phi^3-\frac{\epsilon}{2}\phi^2\right)  \ d\bm{x}.
\end{split}
\end{equation}
The PFC model takes the form
\begin{equation}\label{the4-3}
\begin{split}
&\frac{\partial \phi}{\partial t}=\Delta \mu,\\
&\mu=\frac{\delta E}{\delta \phi}=(1+\Delta)^2\phi+\phi^3-\phi^2-\epsilon\phi.
\end{split}
\end{equation}
The computational domain is set as $\Omega=[0,8]^2$ with periodic boundary conditions. The initial condition is given by $\phi_0(x,y)=\sin(\frac{\pi x}{4})\sin(\frac{\pi y}{4})$, and the parameter is set to $\epsilon=0.02$ to test the accuracy of the RAV scheme. Following the same procedure as in Example 1, Table \ref{table2} shows that the proposed RAV scheme achieves second‑order accuracy in time. These numerical results are consistent with the error estimates in Theorem \ref{theorem3-2}.

\begin{table}[H]
\caption{Errors and convergence rates for $\phi$ of the second-order RAV scheme for the PFC equation at $T=0.016$.}\label{table2}
\centering
\scalebox{0.65}{
\resizebox{\linewidth}{!}{
\begin{tabular}{c c c c c c} \hline  		
 $\Delta t$   &  $L^2$-error  &  Order    &  $L^\infty$-error  &  Order  & $\max_{i\leq n} r^i$  \\ \hline
 $1.6e-3$		& 2.031e-04    &  & 8.707e-04  &  &  0.0 \\ [6pt]
 $8e-4$	& 5.573e-05  &  1.87 &  2.563e-04    & 1.76 &  0.0 \\  [6pt]
$4e-4$	& 1.502e-05    &   1.89  & 7.391e-05  &  1.79 &  0.0 \\  [6pt]
$2e-4$	& 3.938e-06     &  1.94 &   2.036e-05    &  1.86 &  0.0 \\  [6pt]
$1e-4$	& 9.973e-07      &  1.98 &   5.493e-06    &  1.89 & 0.0 \\  \hline 
\end{tabular}}
	}
\end{table}

We next investigate the performance of the RAV and SAV schemes for the PFC model with large time steps. The parameters are chosen as $\epsilon=0.02$, $\lambda=0.2$, $T=40$, and the computational domain is taken as $\Omega=[0,128]^2$, with the initial condition given by $\phi_0(x,y)=0.1+0.1\text{Rand}(x,y)$. $\text{Rand}(x,y)$ is the random number in $[-1,1]$ with zero mean. Figure \ref{figure4} and \ref{figure5} present the density fields computed using the SAV and RAV methods with different time steps, as well as the differences between the auxiliary and original variables. It is evident that the RAV scheme maintains strict consistency between the original and auxiliary variables, while the SAV scheme shows an error that accumulates over time. Furthermore, 
from Figure \ref{figure6} we observe $Q^{n+1}-r^n\geq 0$, which implies that the original and modified energies remain consistent and decay monotonically.

\begin{figure}[h]
	\centering
	\subfigure[$\Delta t=1$]{
		\includegraphics[scale=0.16]{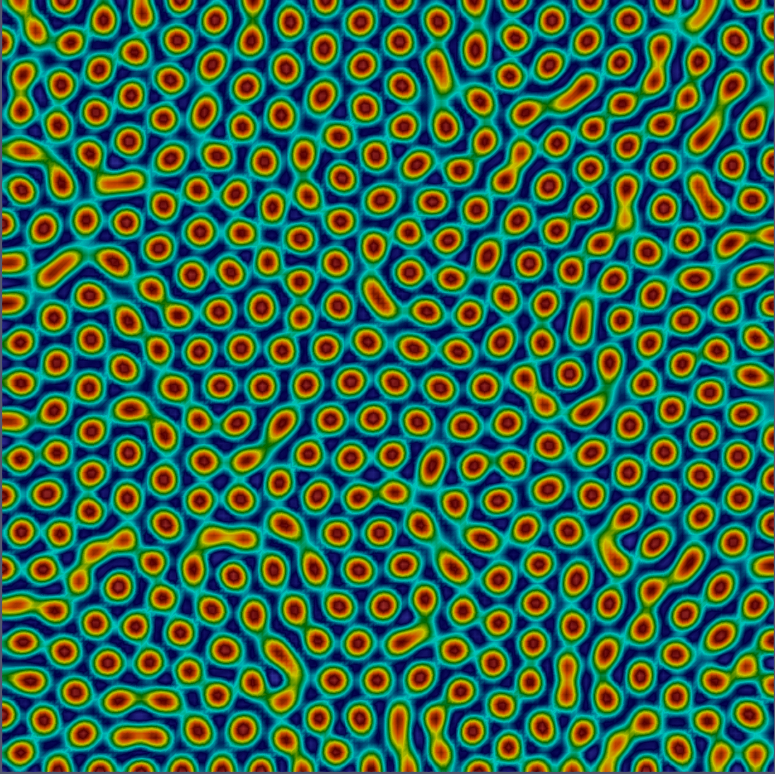}
	}
	\subfigure[$\Delta t=1/2$]{
		\includegraphics[scale=0.16]{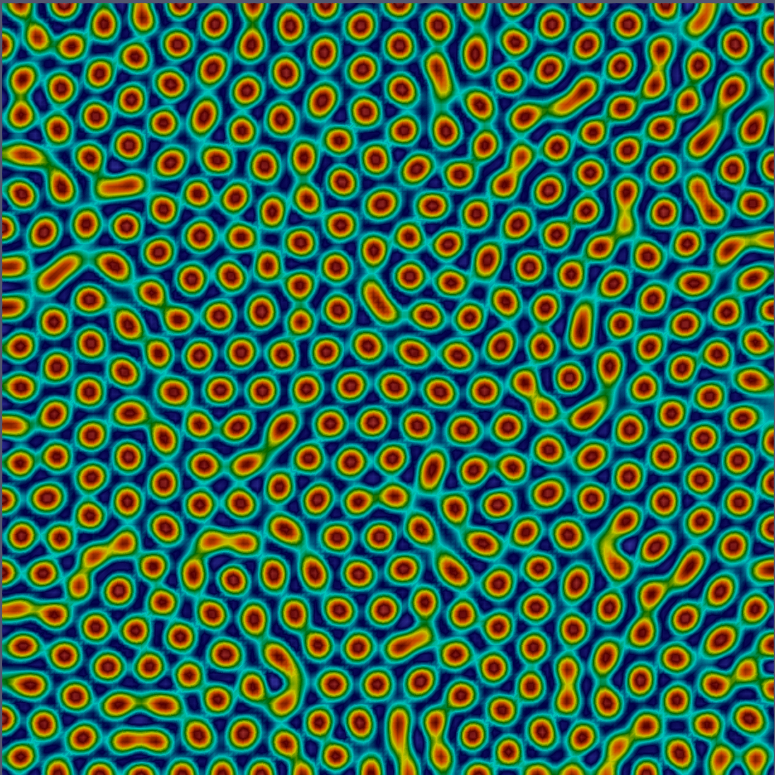}
	}
	\subfigure[$\Delta t=1/4$]{
		\includegraphics[scale=0.16]{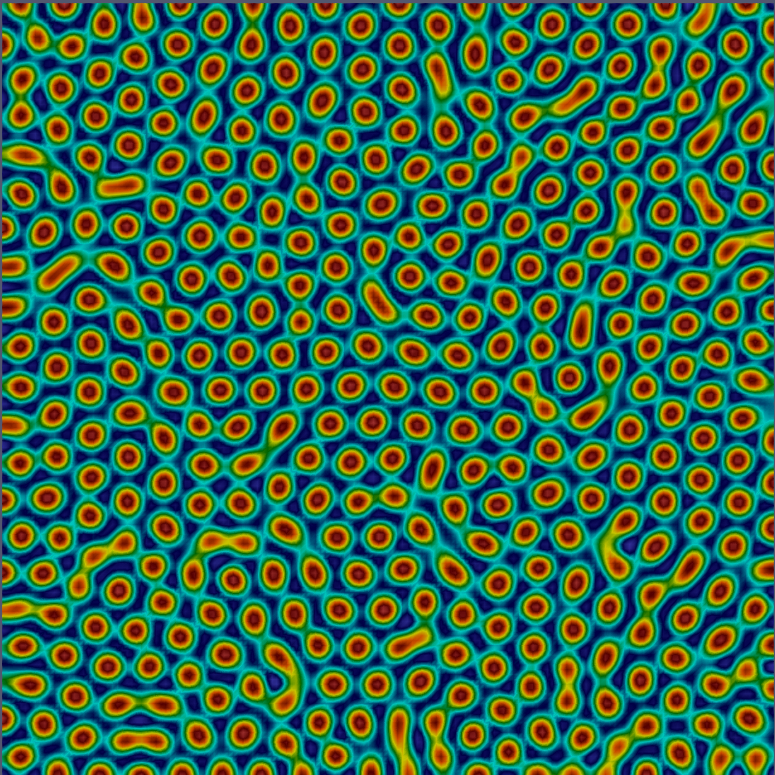}
	}
	\subfigure[$|r_{\text{sav}}-\sqrt{\int_\Omega F(\phi)d\bm{x}+C_0}|$]{
		\includegraphics[scale=0.16]{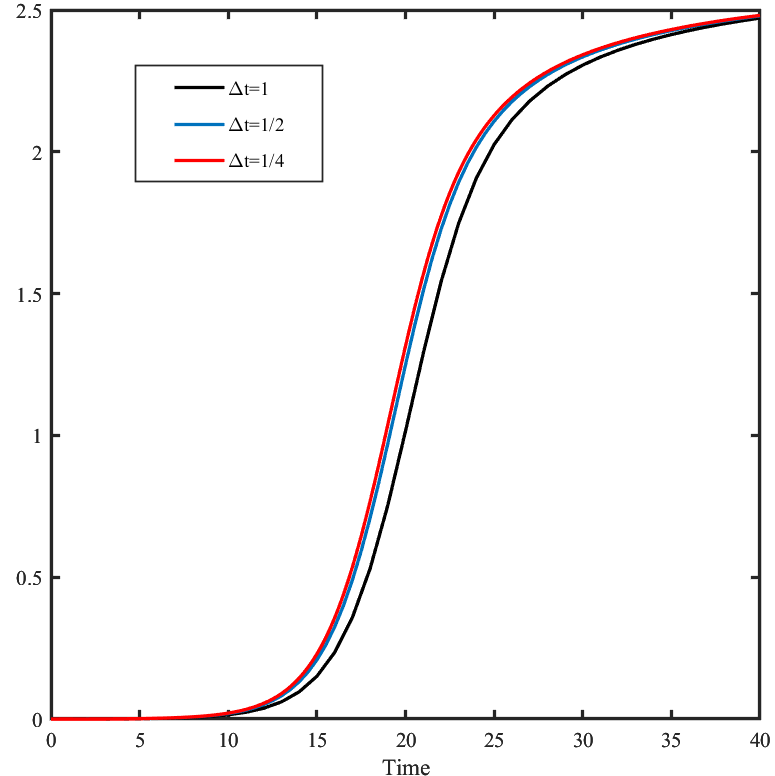}
	}
	
	\caption{Snapshots of the phase variable $\phi$ computed by the SAV–CN scheme at $t=40$. The line graphs give the discrepancy between the auxiliary variable and the original variable.}
	\label{figure4}
\end{figure}

\begin{figure}[h]
	\centering
	\subfigure[$\Delta t=1$]{
		\includegraphics[scale=0.16]{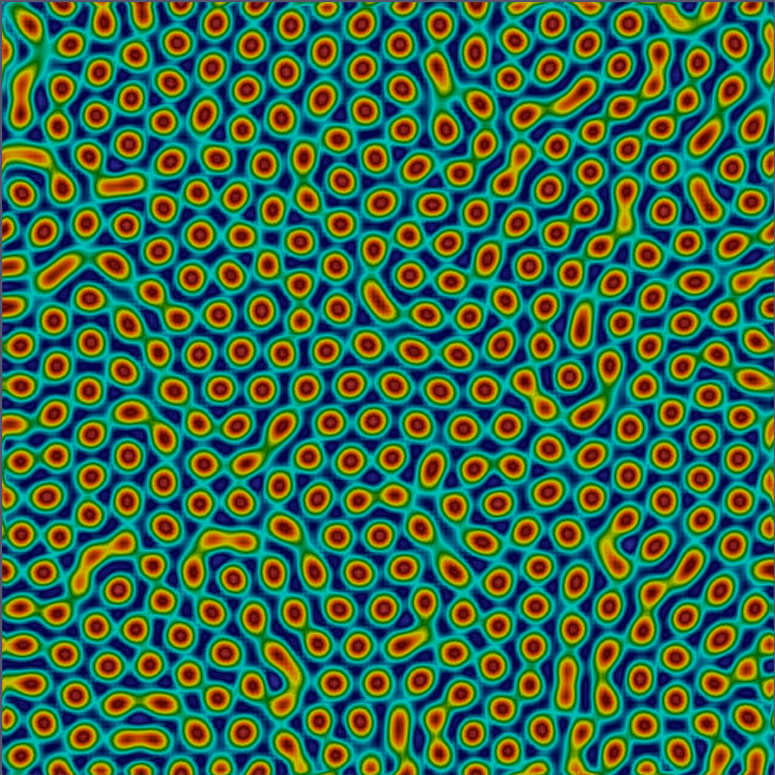}
	}
	\subfigure[$\Delta t=1/2$]{
		\includegraphics[scale=0.16]{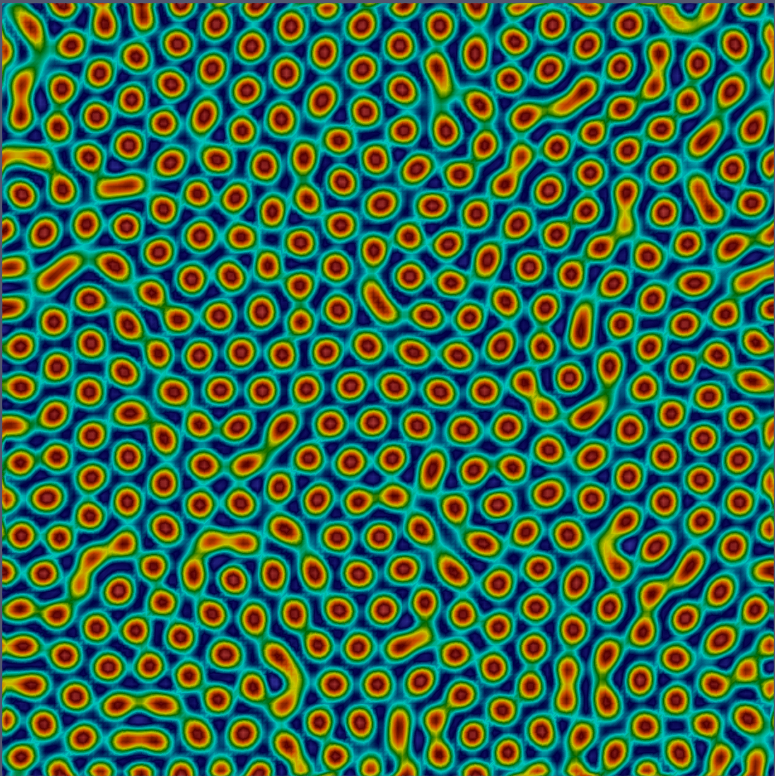}
	}
	\subfigure[$\Delta t=1/4$]{
		\includegraphics[scale=0.16]{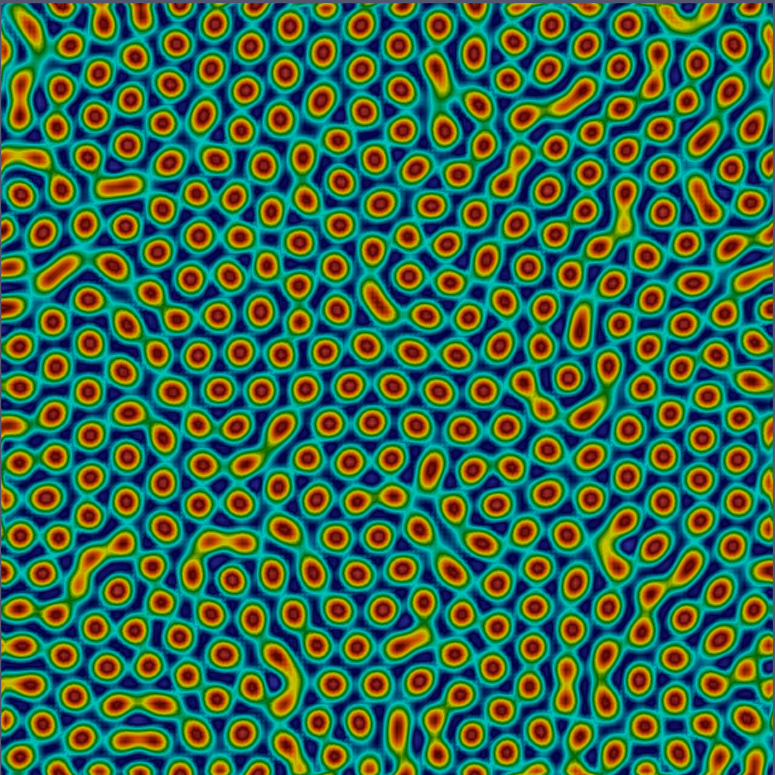}
	}
	\subfigure[$r_{\text{rav}}-0$]{
		\includegraphics[scale=0.16]{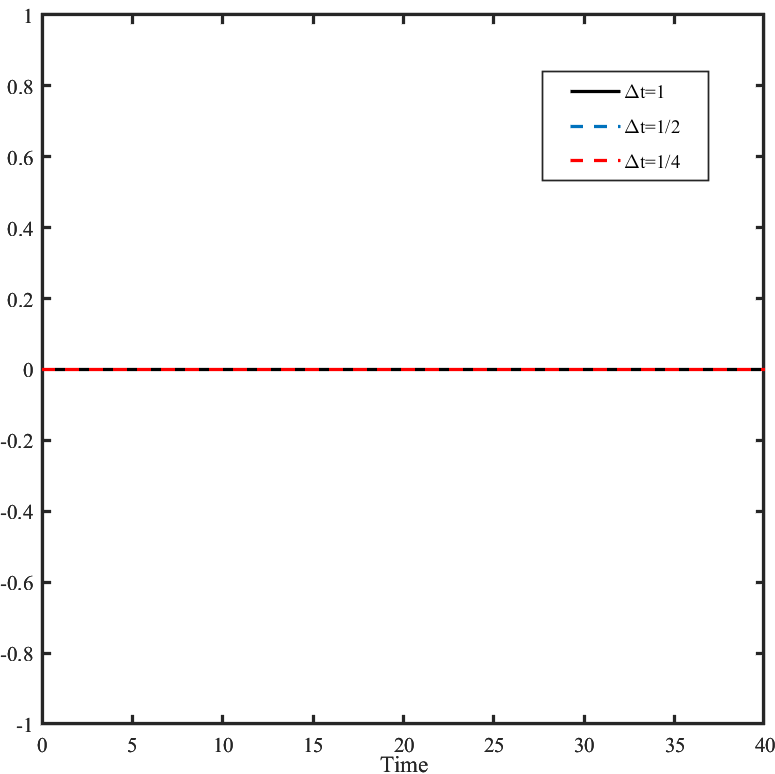}
	}
	
	\caption{Snapshots of the phase variable $\phi$ computed by the RAV scheme (\ref{the2-3}) at $t=40$. The line graphs give the discrepancy between the auxiliary variable and the original variable.}
	\label{figure5}
\end{figure}

\begin{figure}[H]
	\centering
	\subfigure[$(Q^{n+1}-r^n)$ vs. time]{
		\includegraphics[scale=0.38]{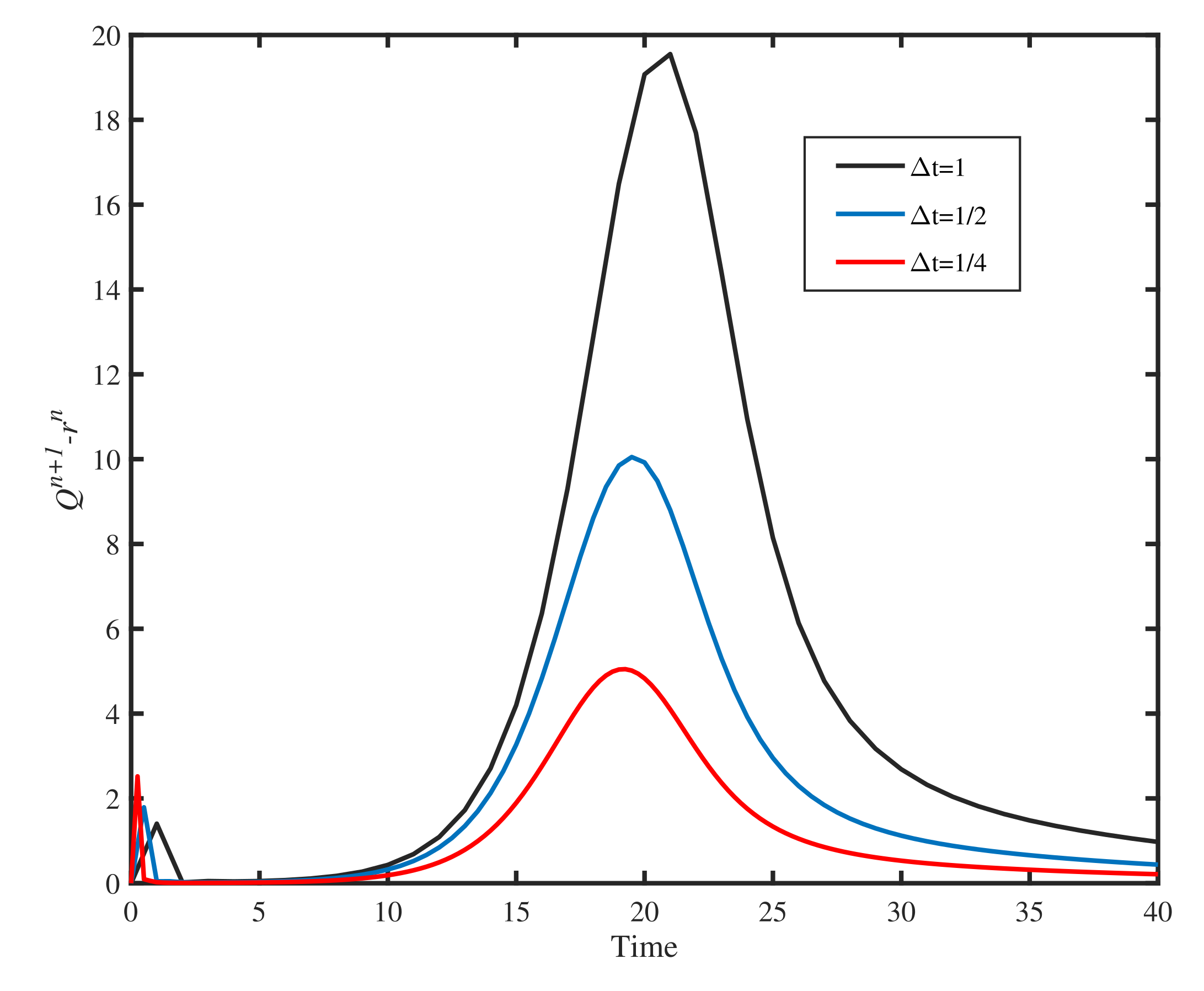}
	}
	\subfigure[Energy vs. time]{
		\includegraphics[scale=0.38]{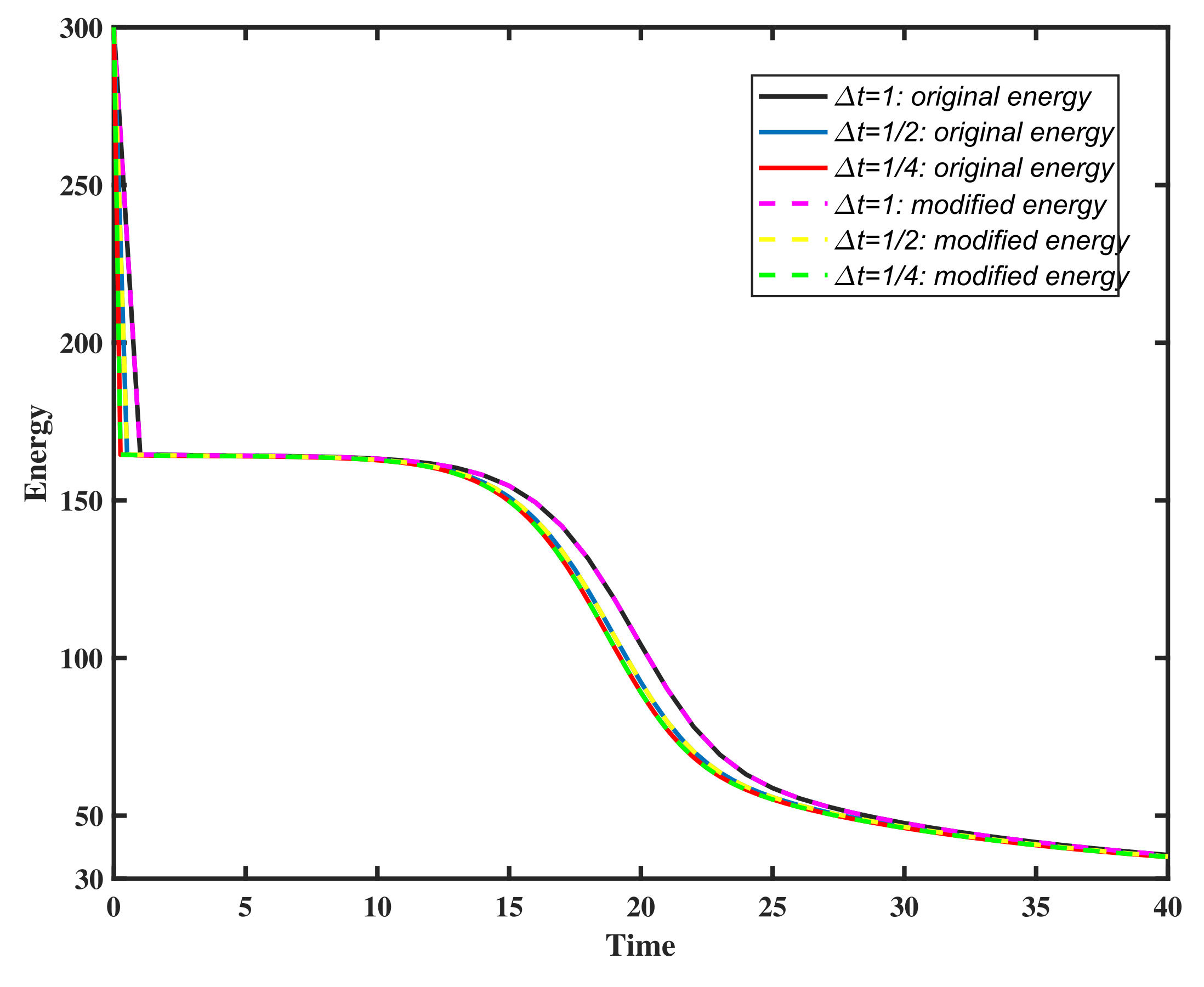}
	}
	\caption{Evolution of total energy for the RAV scheme with different time steps.}
	\label{figure6}
\end{figure}

\subsection{Phase-field vesicle model}
Vesicles in living organisms exhibit a variety of equilibrium shapes, and their mathematical modeling \cite{du2004phase} and simulation \cite{shen2022energy} have been extensively studied. Due to the strong nonlinearity of the model, achieving efficient and accurate simulation is challenging. In this subsection, the proposed RAV method is employed to simulate the morphological evolution of vesicles in two different environments.
The free energy of the phase-field vesicle model \cite{chen2015decoupled} is
\begin{equation}\label{the4-4}
\begin{split}
E[\phi]&=\frac{\lambda\epsilon}{2}\int_{\Omega}\left(|\Delta \phi|^2-\frac{2}{\epsilon^2}|\nabla\phi|^2+\frac{6}{\epsilon^2}\phi^2|\nabla\phi|^2+\frac{1}{\epsilon^4}(F'(\phi))^2\right) \ d\bm{x}\\
&+\frac{\lambda M_1}{2}(A(\phi)-A(\phi_0))^2+ \frac{\lambda M_2}{2}(B(\phi)-B(\phi_0))^2 ,
\end{split}
\end{equation}
where the bulk energy density is given by $F(\phi)=\frac{1}{4}(\phi^2-1)^2$. The parameter $\lambda$ denotes the surface tension coefficient, and $M_1$ and $M_2$ are penalty parameters. Note that $\frac{1}{2}A(\phi)$ and $\frac{3}{2\sqrt{2}}B(\phi)$ represent the volume and the surface area of the vesicle, where
\begin{equation}\nonumber
\begin{split}
A(\phi)=\int_\Omega (\phi+1) \ d\bm{x}, \ B(\phi)=\int_\Omega (\frac{\epsilon}{2}|\nabla\phi|^2+\frac{1}{\epsilon}F(\phi))\ d\bm{x}.
\end{split}
\end{equation}
Following the $L^2$ gradient flow (the Allen–Cahn dynamics), the governing system can be written as
\begin{equation}\label{the4-5}
\begin{split}
&\frac{\partial \phi}{\partial t}=-\mu,\\
&\mu=\lambda \epsilon\Delta^2\phi+\frac{2\lambda}{\epsilon}\Delta \phi+\frac{6\lambda}{\epsilon}(\phi|\nabla\phi|^2-\nabla\cdot(\phi^2\nabla\phi))+\frac{\lambda}{\epsilon^3}F'(\phi)F''(\phi)\\
&+\lambda M_1(A(\phi)-A(\phi_0))+\lambda\epsilon M_2(B(\phi)-B(\phi_0))(-\Delta\phi+\frac{1}{\epsilon^2}F'(\phi)).
\end{split}
\end{equation}
The computational domain is set as $\Omega=[0,2\pi]^2$ with the homogeneous Neumann boundary condition. Model parameters are configured with $\lambda=1e-3$, $M_1=M_2=5e4$, and $\epsilon=0.1$. The initial condition of the elliptical vesicle is specified by 
\begin{equation}\nonumber
\begin{split}
\phi_0(x,y)=\text{tanh}\left(\frac{0.35\pi-\sqrt{(x-\pi)^2/0.35+(y-\pi)^2/1.5}}{\sqrt{2}\epsilon}\right).
\end{split}
\end{equation}

We first consider vesicle evolution under the constraints of volume and surface area conservation. As shown in Figure \ref{figure7}, the vesicle evolves from its initial elliptical profile into a contracted red blood cell (RBC) morphology. Furthermore, Figure \ref{figure8} demonstrates that the volume and surface area are conserved during the entire evolution, and the dynamics satisfy the energy dissipation law. We then simulate the vesicle evolution under an imposed shear flow. The advection term $\bm{u}\cdot \nabla\phi$ is incorporated into equation (\ref{the4-5}), where $\bm{u}=\left(0.1(y-\pi), 0\right)$. As depicted in Figure \ref{figure9}, the initially elliptical vesicle exhibits typical shear‑induced deformation, including noticeable tilting and stretching, while its volume and surface area are approximately preserved during the evolution.

\begin{figure}[H]
\centering
	\subfigure[$t=0$]{
		\includegraphics[scale=0.16]{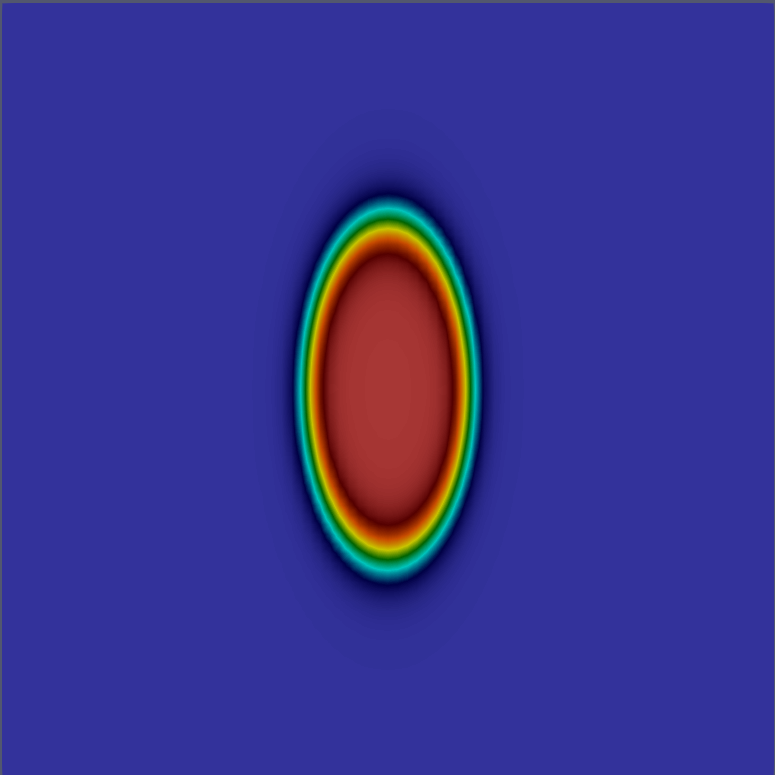}
	}
	\subfigure[$t=20$]{
		\includegraphics[scale=0.16]{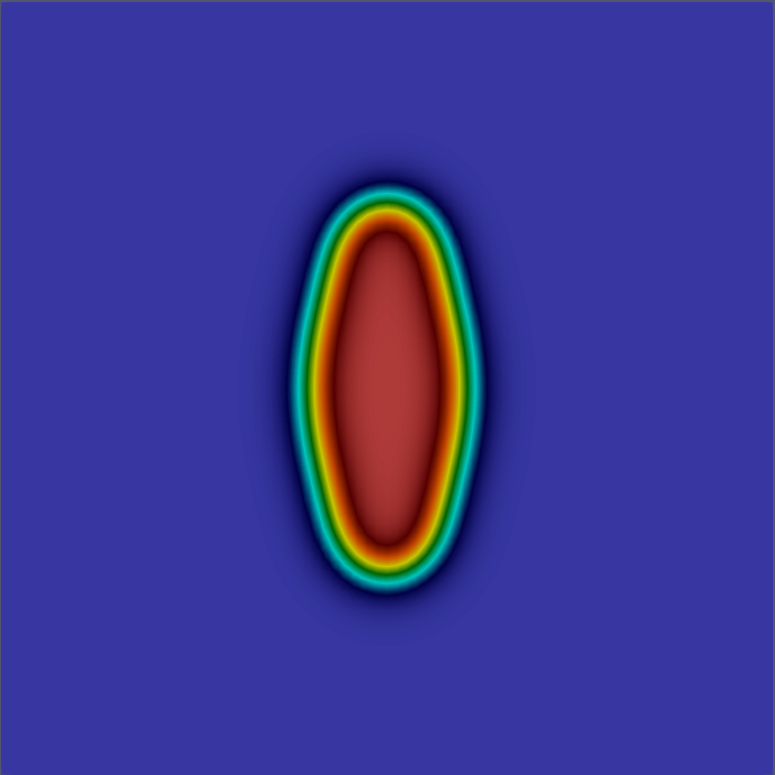}
	}
	\subfigure[$t=100$]{
		\includegraphics[scale=0.16]{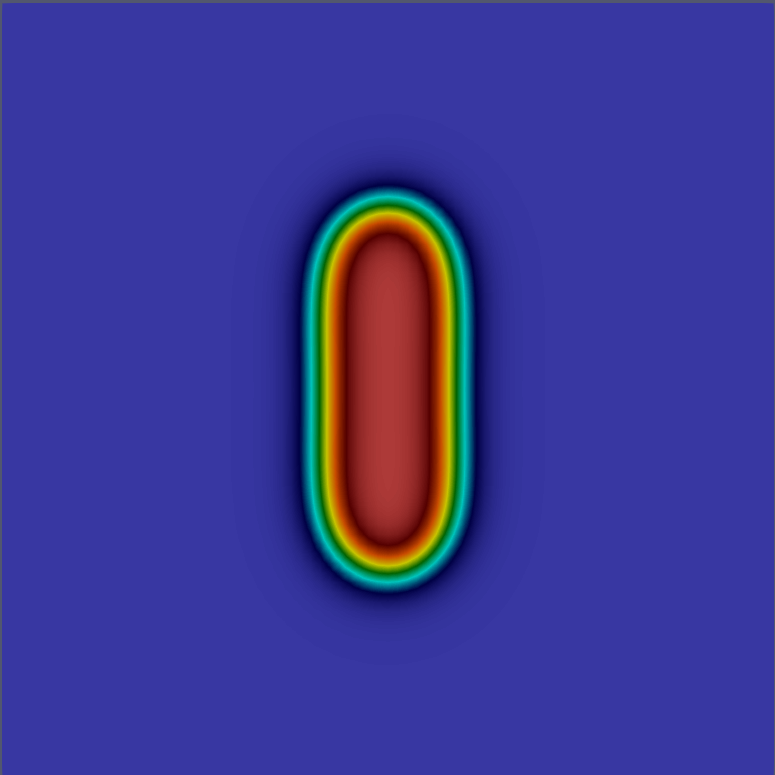}
	}
	\subfigure[$t=150$]{
		\includegraphics[scale=0.16]{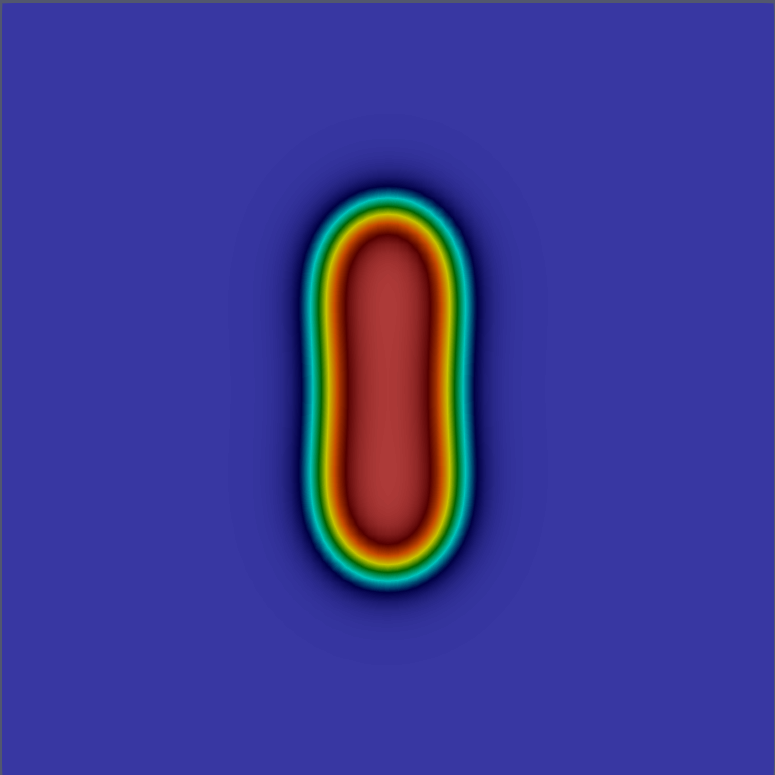}
	}
	
\caption{Evolution of a vesicle at different times.}
\label{figure7}
\end{figure}

\begin{figure}[H]
\centering
	\subfigure[Volume Difference (VD) vs. time]{
		\includegraphics[scale=0.29]{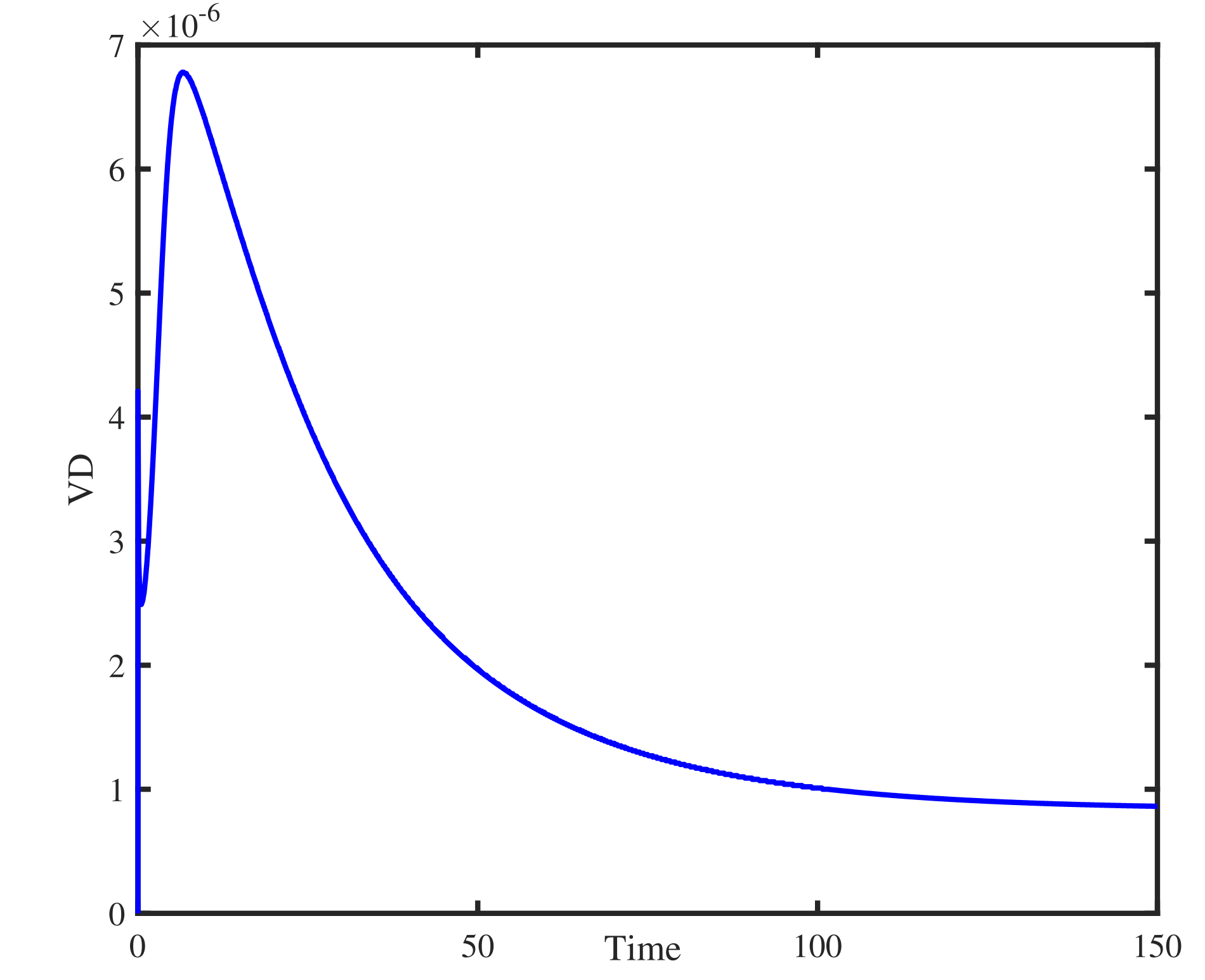}
	}
	\subfigure[Surface Area Difference (SAD) vs. time]{
		\includegraphics[scale=0.29]{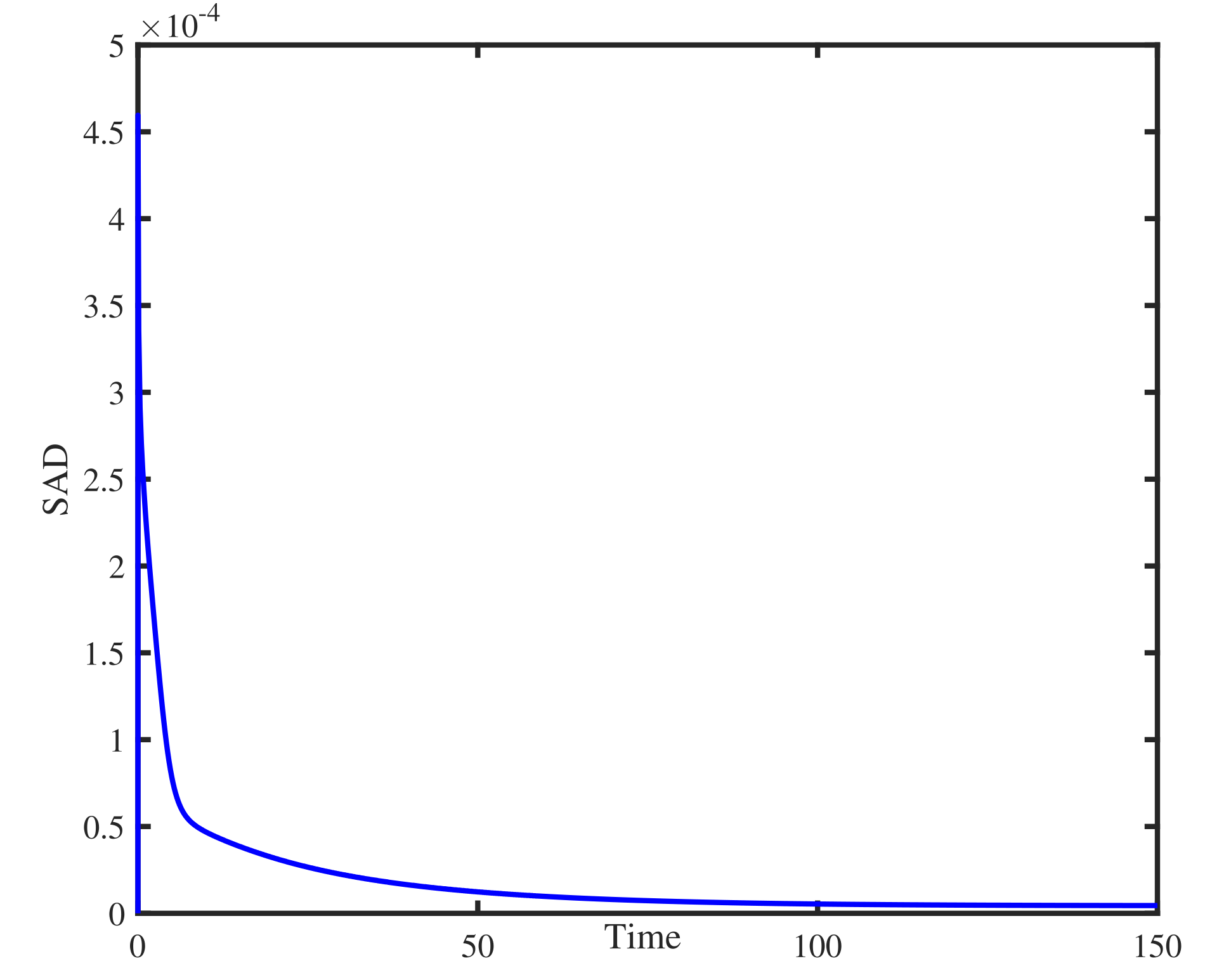}
	}
	\subfigure[Energy vs. time]{
		\includegraphics[scale=0.29]{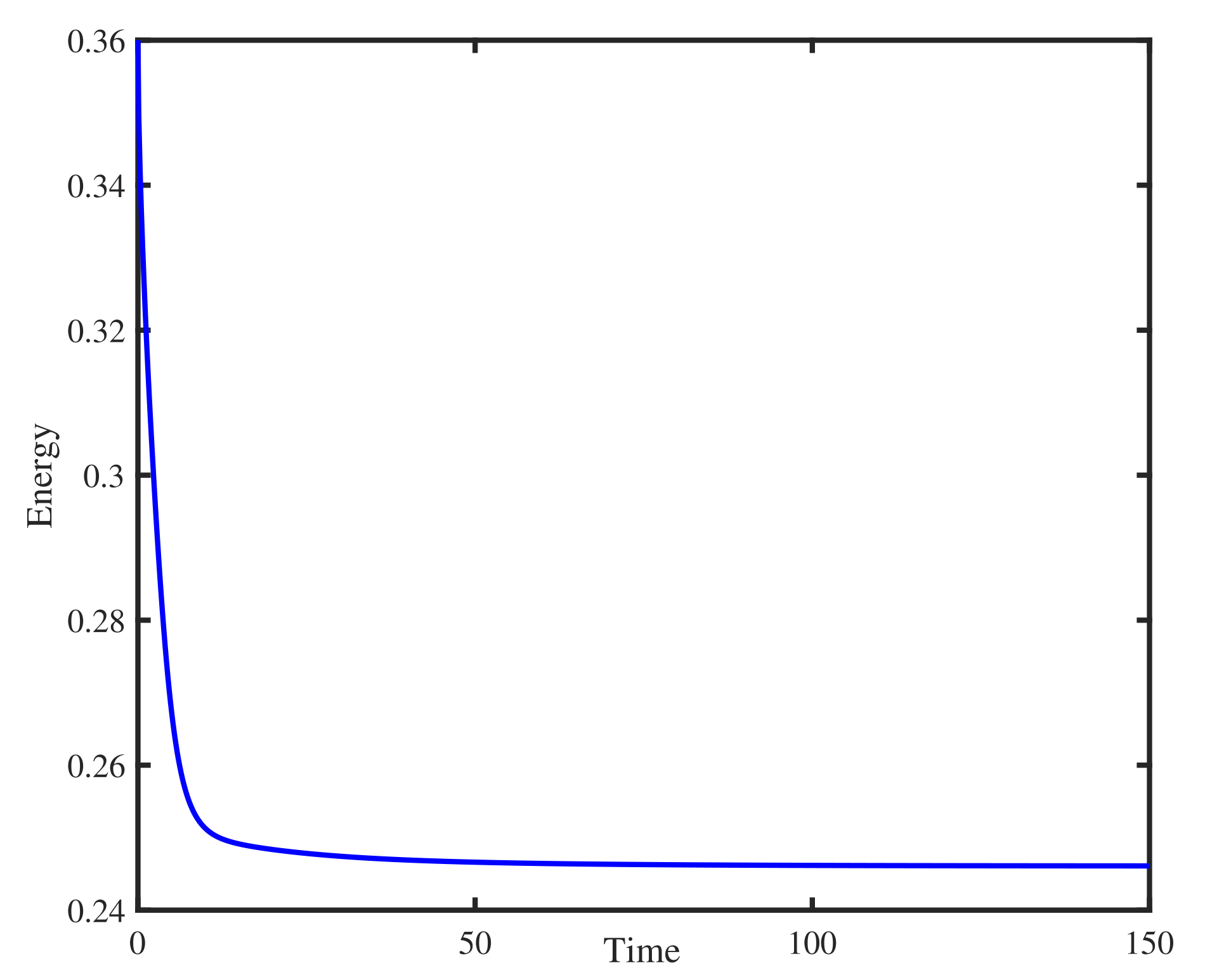}
	}
	
\caption{Volume difference, surface area difference, and energy evolution curves.}
\label{figure8}
\end{figure}

\begin{figure}[H]
\centering
	\subfigure[$t=1.5$]{
		\includegraphics[scale=0.16]{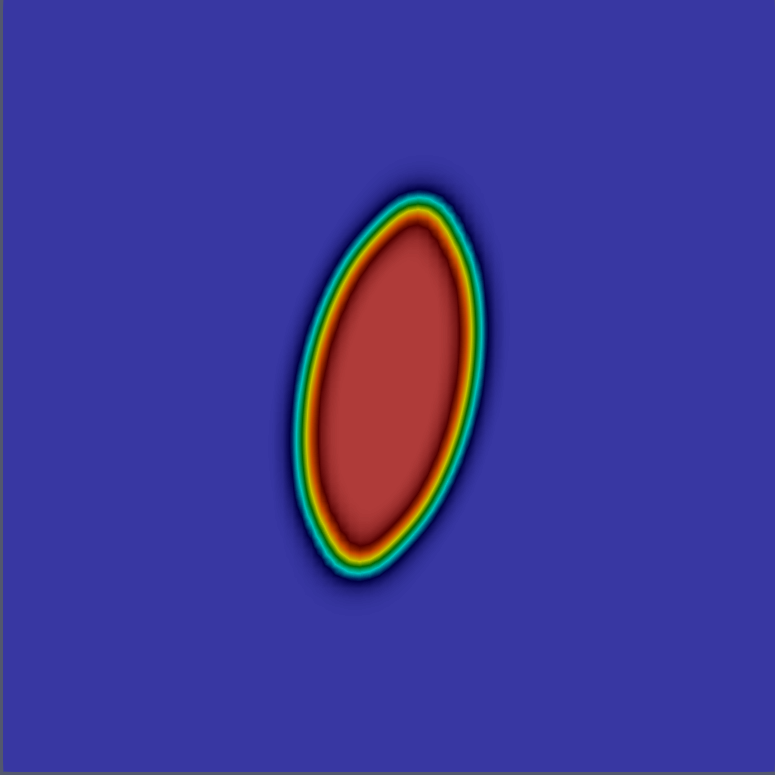}
	}
	\subfigure[$t=3.5$]{
		\includegraphics[scale=0.16]{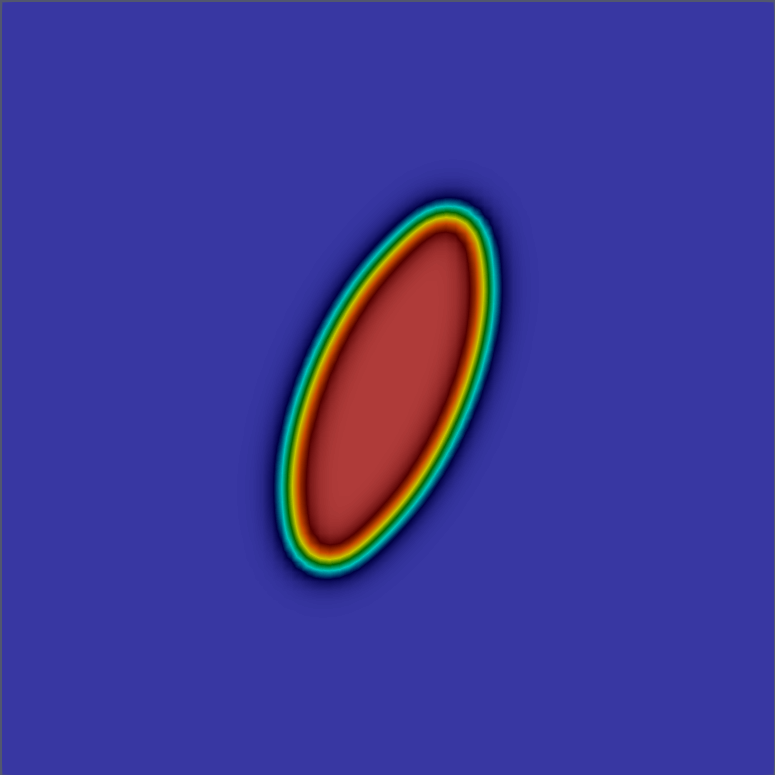}
	}
	\subfigure[$t=6.5$]{
		\includegraphics[scale=0.16]{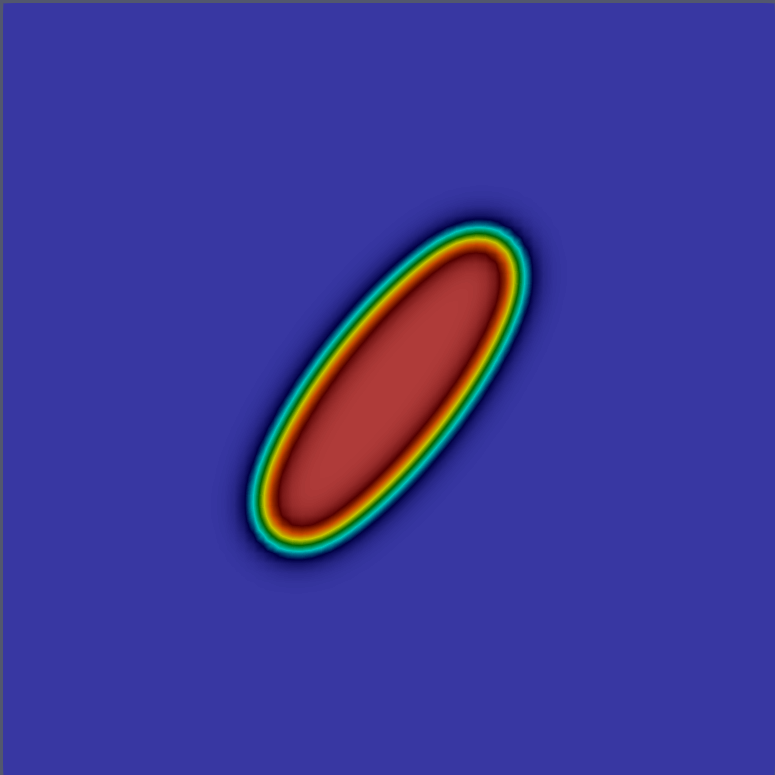}
	}
	\subfigure[VD and SAD vs. time]{
		\includegraphics[scale=0.16]{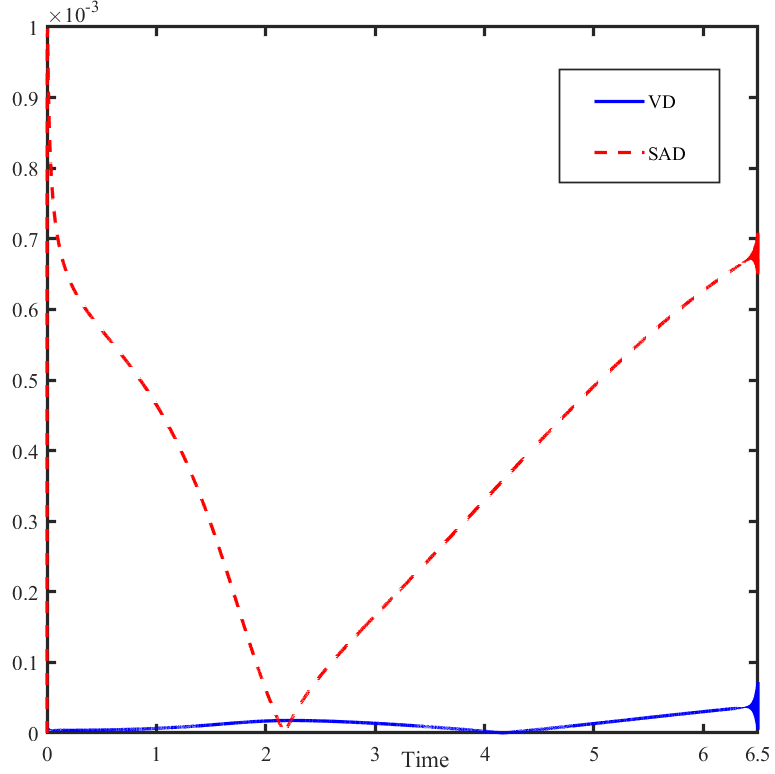}
	}
	
\caption{Evolution of a vesicle at different times.}
\label{figure9}
\end{figure}

\subsection{Phase-field surfactant model}
Two-phase systems with surfactants have extensive applications in scientific and industrial fields, such as oil recovery\cite{spinler2000enhancement} and food processing \cite{myers2020surfactant}, etc. This example focuses on the phase‑field surfactant model \cite{xu2021efficient}, which is defined by the total energy functional
\begin{equation}\label{the4-6}
\begin{split}
E[\phi] = \int_\Omega \left(\frac{1}{2}|\nabla\phi|^2+F(\phi)+\frac{1}{2}|\nabla\rho|^2+G(\rho)-\frac{\gamma_1}{2}\rho|\nabla\phi|^2+\frac{\gamma_2}{4}|\nabla\phi|^4\right) \ d\bm{x},
\end{split}
\end{equation}
where $F(\phi)=\frac{1}{4\epsilon^2}(\phi^2-1)^2$ and $G(\rho)=\frac{1}{4\delta^2}\rho^2(\rho-1)^2$. By the energy variational method, we obtain the coupled surfactant model, which reads as
\begin{equation}\label{the4-7}
\begin{split}
&\frac{\partial \phi}{\partial t}=M_{\phi}\Delta \mu_\phi,\\
&\mu_\phi=-\Delta \phi+F'(\phi)+\gamma_1 \nabla\cdot\left(\rho\nabla\phi \right)-\gamma_2 \nabla\cdot\left(|\nabla\phi|^2\nabla\phi\right), \\
&\frac{\partial \rho}{\partial t}=M_{\rho} \Delta \mu_\rho,\\
&\mu_\rho=-\Delta \rho+G'(\rho)-\frac{\gamma_1}{2}|\nabla\phi|^2.
\end{split}
\end{equation}
The computational domain $\Omega=[0,2\pi]^2$ with the homogeneous Neumann boundary condition. The parameters $M_{\phi}=M_{\rho}=2e-3$, $\epsilon=\delta=0.08$, $\gamma_1=0.5$ and $\gamma_2=1e-4$.

We investigate the spinodal decomposition dynamics of a homogeneous binary mixture that is quenched into the unstable domain of its miscibility gap. Two different random fields are used as the initial conditions, corresponding to $\phi_0(x,y)=0.01\text{Rand}_1(x,y)$ and $\rho_0(x,y)=0.2+0.01\text{Rand}_2(x,y)$. The $\text{Rand}(\cdot)$ is the random number in $[-1,1]$ and has zero mean. We set the time step $\Delta t=0.01$ for the simulation. As shown in Figure \ref{figure10}, the concentration variable $\phi$ evolves from an initially disordered state into well‑defined phase‑separated regions, while the concentration variable $\rho$ is driven toward the interfaces of $\phi$ and accumulates along them. Moreover, the energy decays over time, and the original and auxiliary variables remain highly consistent, which indicates the accuracy of the simulation.

\begin{figure}[H]
\centering
	\subfigure[$t=0.1$]{
		\includegraphics[scale=0.17]{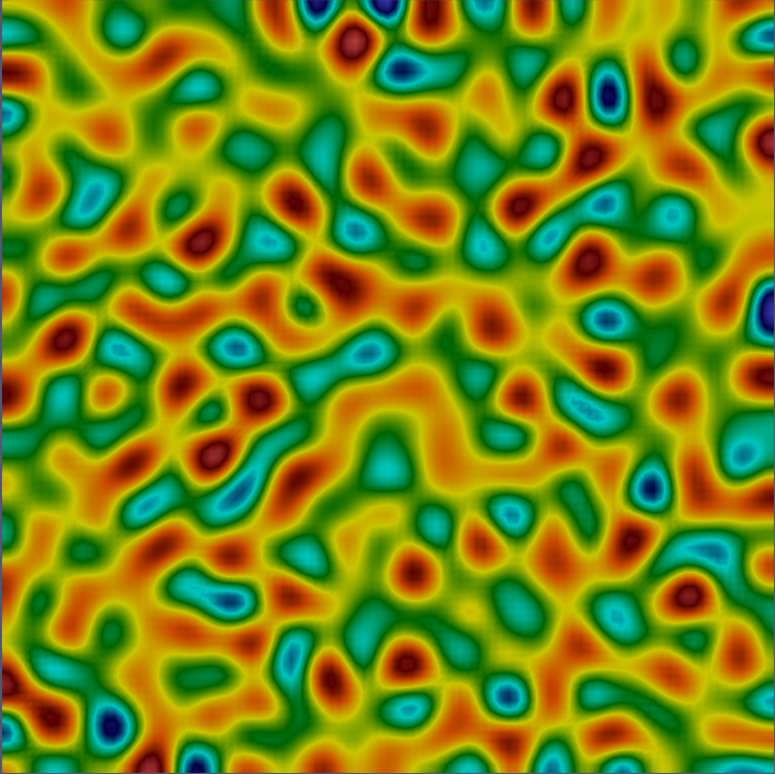}
		\includegraphics[scale=0.17]{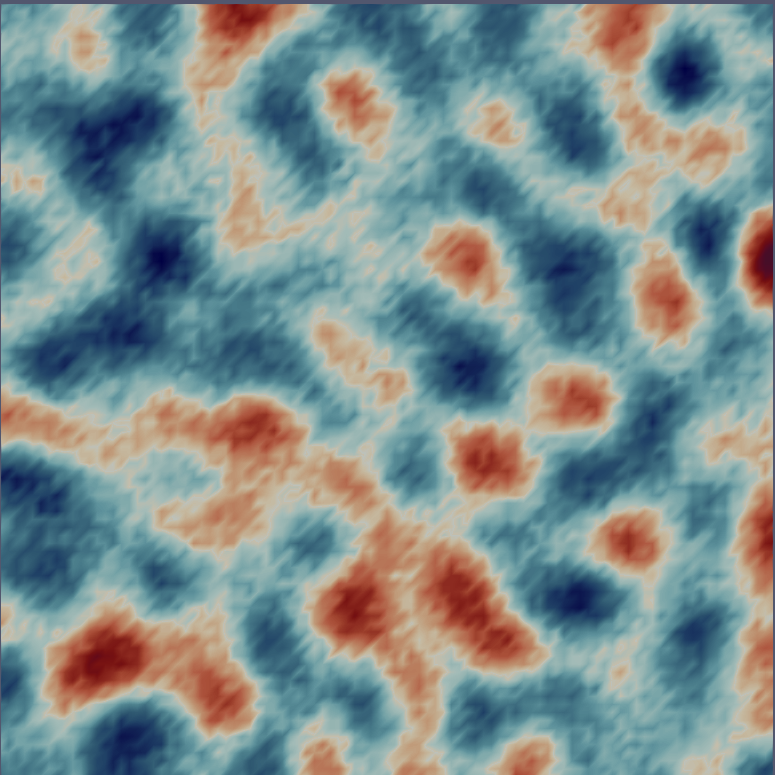}
	}
	\subfigure[$t=3.0$]{
		\includegraphics[scale=0.17]{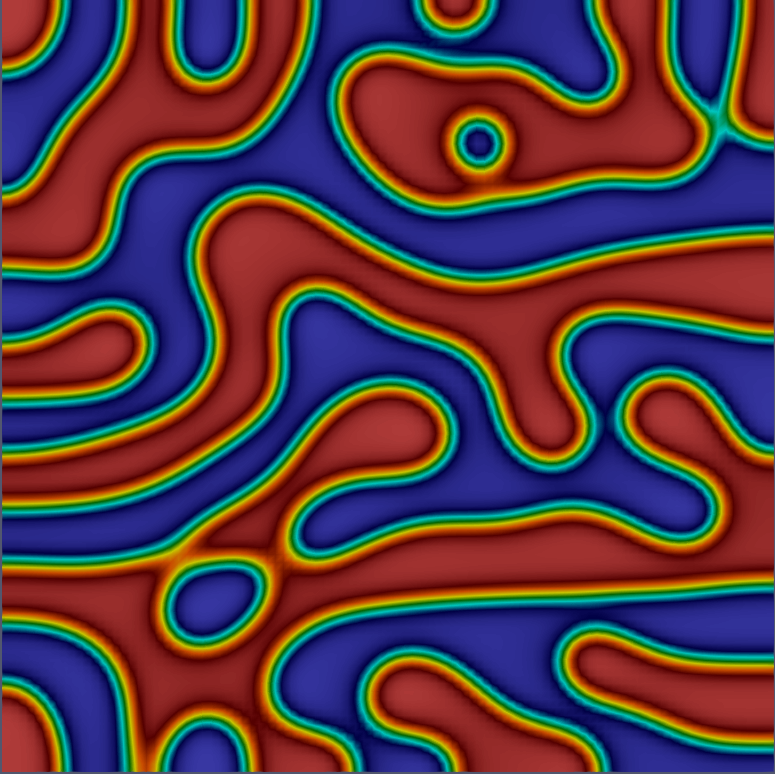}
		\includegraphics[scale=0.17]{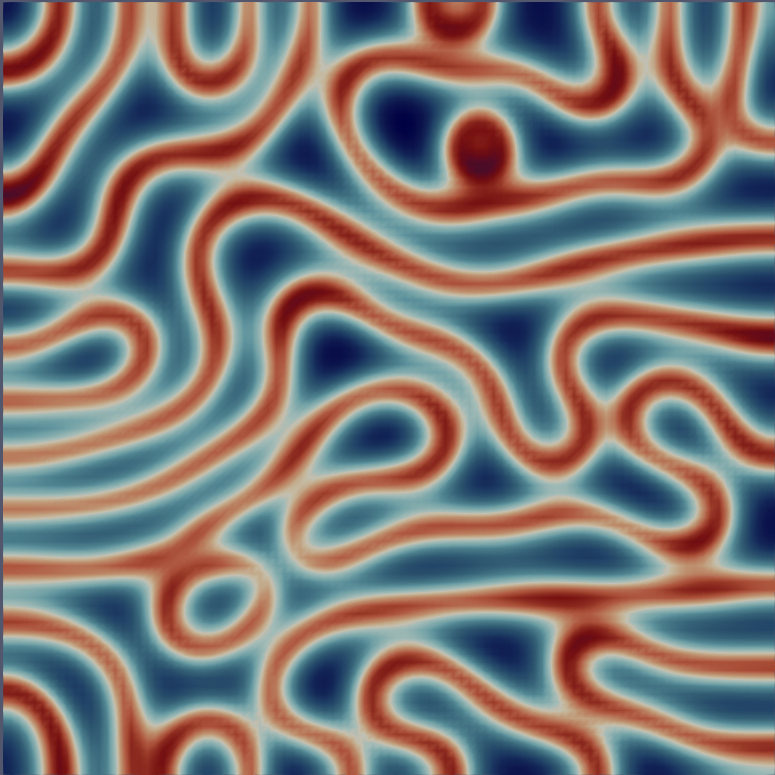}
	}
	\subfigure[$t=5.0$]{
		\includegraphics[scale=0.17]{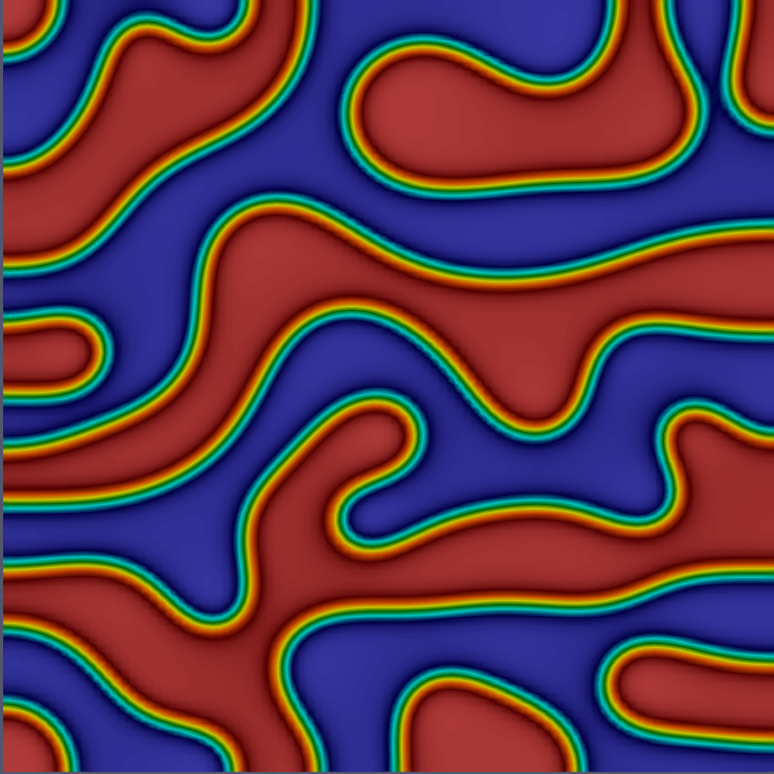}
		\includegraphics[scale=0.17]{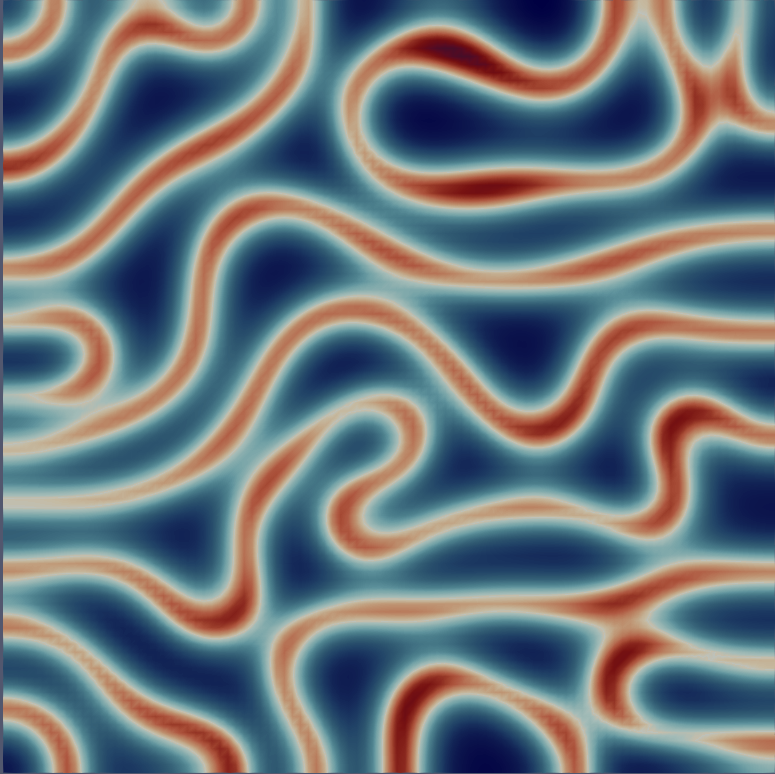}
	}
	\subfigure[Energy and $r-0$ vs. time]{
		\includegraphics[scale=0.17]{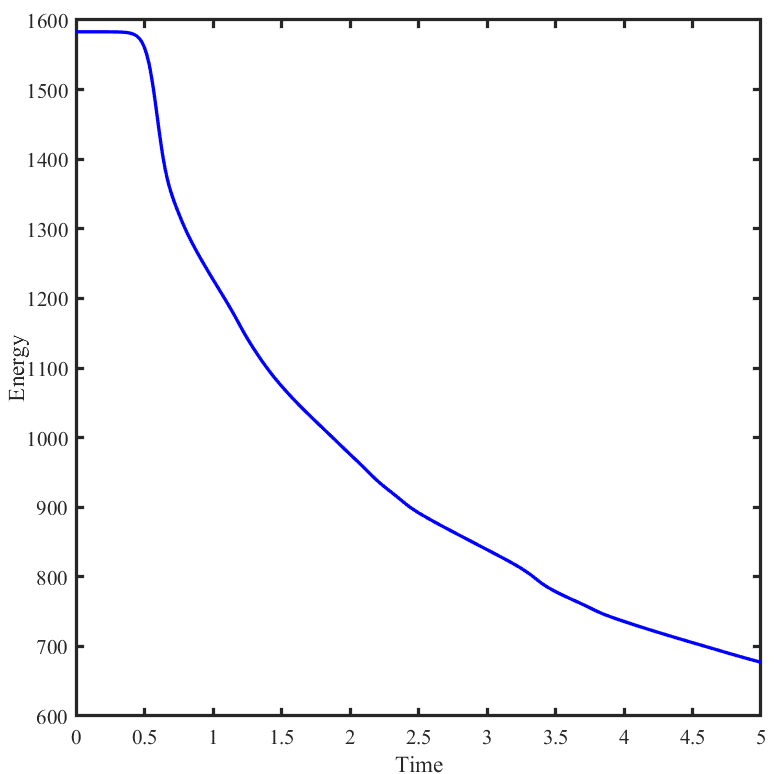}
		\includegraphics[scale=0.17]{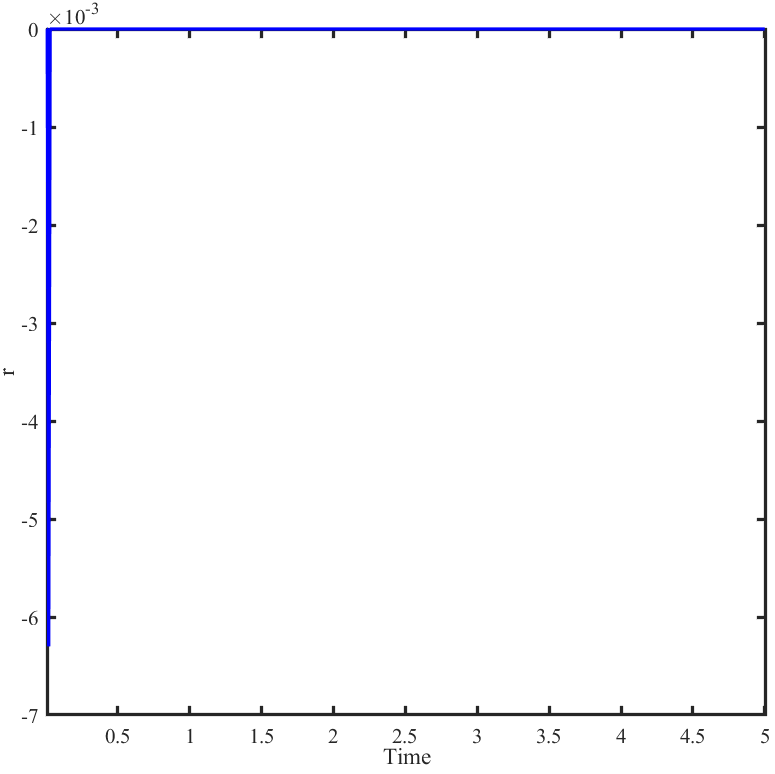}
	}
	
\caption{Time evolution of spinodal decomposition and Energy dissipation.}
\label{figure10}
\end{figure}

\section{Conclusions and remarks}
\label{section5}
In this work, we presented a regularized auxiliary variable method to deal with the inconsistency between the auxiliary system and the original system in numerical computation of existing auxiliary variable methods, thus improving the accuracy and stability of the numerical solutions. The RAV scheme is unconditionally energy stable, and it provides an intuitive characterization of the relation between the modified energy and the original energy (Theorem \ref{theorem2-2}). More importantly, we establish for the first time a rigorous error estimate for this auxiliary variable scheme without imposing any restriction on the time step, while retaining the same computational cost as the conventional IMEX scheme.

We would like to remark that the RAV approach is not limited to gradient flows. It can be applied to a wide class of nonlinear dissipative systems, such as the Navier–Stokes equations \cite{li2022new}, multiphase flow problems \cite{li2025class}, and kinetic equations \cite{zhang2025sav}. These extensions will be explored in future work.

\section*{Acknowledgments}
Z. Wang and P. Lin are partially supported by the National Natural Science Foundation of China under Grant Nos. 12501535, 12371388, and by the Beijing Natural Science Foundation under Grant No. IS25005.

\bibliographystyle{elsarticle-num}
\bibliography{Ref}

\end{document}